%% file: tepmult-final.tex
\documentclass[12pt]{article}
\usepackage[usenames,dvipsnames]{color}
\usepackage{graphicx,amsthm,amsmath,amssymb,color,srcltx,a4wide,xr,xargs,bbm,natbib}
\usepackage[colorlinks]{hyperref}
\usepackage[shortlabels]{enumitem}

\usepackage{pdfsync}

\input{commands}

\def\anticlustering{finite cluster}
\def\Anticlustering{Finite cluster}

\usepackage{cleveref}

\newtheorem{theorem}{Theorem}[section]
\newtheorem{lemma}[theorem]{Lemma}

\newtheorem{corollary}[theorem]{Corollary}
\newtheorem{proposition}[theorem]{Proposition}
\newtheorem{definition}[theorem]{Definition}
\newtheorem{hypothesis}[theorem]{Assumption}
\theoremstyle{remark}
\newtheorem{remark}[theorem]{Remark}
\newtheorem{example}[theorem]{Example}

\crefname{hypothesis}{Assumption}{Assumptions}
\crefname{theorem}{Theorem}{Theorems}

\numberwithin{equation}{section}

\begin{document}
\title{The tail empirical process of regularly varying functions of geometrically ergodic Markov
  chains}

\author{Rafa{\l} Kulik\thanks{ Department of Mathematics and Statistics, University of Ottawa, K1N
    6N5 Ottawa, Ontario Canada } \and Philippe Soulier\thanks{Universit\'e Paris Nanterre,
    D{\'e}partement de math{\'e}matique et informatique, Laboratoire MODAL'X, 200 avenue de la
    R{\'e}publique, 92000 Nanterre, France} \and Olivier Wintenberger\thanks{Sorbonne Universit\'es,
    LSTA, 4 place Jussieu, BP 158 75005 Paris, FRANCE} }

\maketitle

\begin{abstract}
  We consider a stationary regularly varying time series which can be expressed as a function of a
  geometrically ergodic Markov chain. We obtain practical conditions for the weak convergence of the
  tail array sums and feasible estimators of cluster statistics. These conditions include the
  so-called geometric drift or Foster-Lyapunov condition and can be easily checked for most usual
  time series models with a Markovian structure. We illustrate these conditions on several models
  and statistical applications.  A counterexample is given to show a different limiting behavior
  when the geometric drift condition is not fulfilled.
\end{abstract}

%\tableofcontents
\section{Introduction}
\label{sec:introduction}

Let $\sequence{X}{j}{\Zset}$ be a stationary, regularly varying univariate time series with marginal
distribution function $F$ and tail index $\alpha>0$. This means that for each integer $h\geq0$, there
exists a non zero Radon measure $\numultseq{0,h}$ on $\uncompact{h+1}$ such that
$\numultseq{0,h}(\overline{\Rset}^{h+1}\setminus\Rset^{h+1}) = 0$ and
\begin{align}
  \lim_{x\to\infty} \frac{\pr((X_0,\dots,X_h) \in x A)}{\pr(X_0>x)} = \numultseq{0,h}(A) \; ,  \label{eq:rv-positive}
\end{align}
for all relatively compact sets $A\subset \uncompact{h+1}$ satisfying $\numultseq{0,h}(\partial
A)=0$. The measure $\numultseq{0,h}$, called the exponent measure of $(X_0,\dots,X_h)$, is
homogeneous with index $-\alpha$, i.e. $\numultseq{0,h}(tA) = t^{-\alpha} \numultseq{0,h}(A)$. The
choice of the denominator $\pr(X_0>x)$ entails that~$\numultseq{0,h}((1,\infty)\times\Rset^h) = 1$
\ie\ that the right tail of the stationary distribution is not trivial and that $X_0$ satisfies the
so-called balanced tail condition.

Many statistical methods for extreme value characteristics of a time series are based on tail array
sums of the form
\begin{align}
  \label{eq:def-tailarray}
  M_n(\phi)  = \frac{1}{n\tail{F}(u_n)} \sum_{j=1}^n  \phi(({X}_{j},\ldots,X_{j+h})/u_n) \; ,
\end{align}
for a fixed non negative integer $h$, a non decreasing sequence $\{u_n\}$ such that
$\lim_{n\to\infty} u_n = \lim_{n\to\infty}n\tail{F}(u_n)=\infty$ and a measurable function $\phi$ on
$\Rset^{h+1}$ such that $\esp[|\phi(({X}_{0},\ldots,X_{h})/u_n)|]<\infty$ for all $n$. An important
example is the tail empirical distribution function, defined by
\begin{align}
  \label{def:usual-tep}
  \widetilde{T}_n(s) =   \frac{1}{n\tail{F}(u_n)} \sum_{j=1}^n  \ind{X_{j}>u_ns}  \; ,
\end{align}
which is used for the estimation of univariate extremal characteristics such as the tail index.  We
are interested in the weak convergence of the centered and renormalized process
\begin{align}
  \label{eq:def-tepmult}
  \TEPmult_{n}(\phi) & = \sqrt{n\bar{F}(u_n)} \{M_n(\phi) - \esp[M_n(\phi)]\}  \; .
\end{align}
There is a huge literature on this problem. Essential references are
\cite{rootzen:leadbetter:dehaan:1998} which investigate weak convergence of tail array
sums,~\cite{drees:1998,drees:2000weighted,drees:2002tail,drees:2003} who developed techniques to
study tail empirical and tail quantile processes for $\beta$-mixing time series, \cite{rootzen:2009}
which reviews the results for the functional convergence of the tail empirical process in the case of
\iid\ and weakly dependent (strong or $\beta$-mixing) univariate time series. Recently,
\cite{drees:rootzen:2010} investigated the weak convergence of $\{\TEPmult_n(\phi),\phi\in \mcg\}$ as a
sequence of random elements in the space $\ell_\infty(\mcg)$, where $\mcg$ is a class of functions.
\cite{drees:segers:warchol:2014} applied the latter reference to the estimation of the empirical
distribution function of the spectral tail process (see definition in \Cref{sec:edfstp}).

In all these references, results are proved under strong or $\beta$-mixing conditions and under
additional assumptions guaranteeing the existence of the limiting variances and tightness in the
case of functional convergence. These additional conditions are notably hard to check and have been verified
for a handful of particular models such as solutions of stochastic recurrence equations (including
some GARCH processes) and certain linear processes such as AR(1) processes. See \eg\
\cite{drees:2000weighted,drees:2003} and \cite{davis:mikosch:2009}.

The main purpose of this paper is to show that for time series which can be expressed as functions
of an underlying Markov chain, these results can be proved under essentially a single, easily
checked condition, namely the geometric or Foster-Lyapunov drift condition. See
\cite[Chapter~15]{meyn:tweedie:2009}.  In order to use it in the context of extreme value theory,
the drift function must have an additional homogeneity property.  This geometric drift condition
was first used in the context of extreme value theory by \cite{roberts:rosenthal:segers:sousa:2006}
who used it to prove that the extremal index is positive. It was later used by
\cite{mikosch:wintenberger:2013} to obtain large deviations and weak convergence to stable laws for
heavy tailed functions of a Markov chain.

The main result of this paper is \Cref{theo:fidi} in \Cref{sec:tail-arrays} which proves the joint
asymptotic normality of tail array sums of the form~(\ref{eq:def-tailarray}) for functions of
irreducible Markov chains which satisfy the geometric drift condition.  In the first place,
irreducibility and the drift condition imply that the time series is $\beta$-mixing with geometric
decay of the $\beta$-mixing coefficients.  Then, it is possible to use the regenerative properties
of irreducible Markov chains to obtain the bounds and other ingredients needed to prove the central
limit theorem, such as the existence of the limiting variance and asymptotic negligibility. The link
between this condition and those used in the literature will be discussed at the end of
\Cref{sec:tail-arrays}.  In \Cref{sec:statapp}, we will strengthen the finite dimensional
convergence to functional convergence over classes of functions. Such a strengthening is needed for
statistical applications.

The geometric drift condition is well known and has been established for many models in the Markov
chain literature, but in order to be useful for extreme value problems, the drift function must have
some homogeneity properties. In \Cref{sec:two-models}, we will provide a practical method to obtain
a suitable drift function.

The geometric decay of the $\beta$-mixing coefficient is actually not an essential ingredient of the
proof of our results. However, as illustrated in \Cref{sec:counterexample}, there are examples of
non geometrically ergodic Markov chains for which the centered and normalized tail empirical process
has a non Gaussian limit, the normalization being different from the usual $\sqrt{n\bar{F}(u_n)}$.
Thus the geometric drift condition cannot be relaxed easily and it is the subject of further
research to find practical sufficient conditions for non geometrically ergodic Markov chains.

The rest of the paper is organized as follows. \Cref{sec:results} contains our main results on
functions of Markov chains, examples and the aforementioned counterexample. \Cref{sec:general}
contains a central limit theorem for tail array sums and \Cref{sec:proof-markov} contains the proof
of the results of \Cref{sec:results}.

\section{Main results for functions of  Markov chains}
\label{sec:results}

Our context is a slight extension of the one in \cite{mikosch:wintenberger:2013}. We now assume
that $\sequence{X}{j}{\Nset}$ is a function of a stationary Markov chain
$\sequence{\mby}{j}{\Nset}$, defined on a probability space $(\Omega,\mcf,\pr)$, with values in a
measurable space $(E,\mce)$. That is, there exists a measurable real valued function $g:E\to\Rset$ such that
$X_j = g(\mby_j)$.

\begin{hypothesis}
  \label{hypo:drift-small}
  \begin{enumerate}[(i),wide=0pt]
  \item The Markov chain $\{\mby_j,j\in\Zset\}$ is strictly stationary under $\pr$.
    \item The sequence $\{X_j=g(\mby_j),j\in\Zset\}$ is regularly varying with tail
    index $\alpha>0$.
  \item There exist a measurable function $V:E\to[1,\infty)$, $\gamma\in (0,1)$ and $b>0$ such that for all $y\in E$,
    \begin{align}
      \label{eq:drift}
      \esp[V(\mby_1) \mid \mby_0=y] \leq \gamma V(y)  + b \;.
    \end{align}
  \item There exist an integer $m\geq1$ and $x_0\geq 1$ such that for all $x\geq x_0$, there exists a probability measure
    $\nu$ on $(E,\mce)$ and $\epsilon>0$ such that, for all $y \in\{V \leq x\}$ and all measurable
    sets $B\in\mce$,
    \begin{align}
       \label{eq:small-set}
       \pr(\mby_m \in B \mid \mby_0=y) \geq \epsilon \nu(B) \; .
    \end{align}
  \item There exist $q_0\in(0,\alpha)$ and a constant $c>0$ such that
    \begin{align}
      \label{eq:g-bound-v}
      |g|^{q_0} \leq c V \; .
    \end{align}
  \item For every $s>0$,
    \begin{align}
      \limsup_{n\to\infty} \frac{1}{u_n^{q_0}\tail{F}(u_n)} \esp\left[V(\mby_0) \ind{|g(\mby_0)|>u_n s}
      \right] < \infty \; .  \label{eq:def-Q}
    \end{align}
  \end{enumerate}
\end{hypothesis}
Under \Cref{hypo:drift-small}, it is well known that the chain $\{\mby_j\}$ is irreducible and
geometrically ergodic and $\esp[V(\mby_0)]<\infty$. This implies that the chain $\{\mby_j\}$ and the
sequence $\{X_j\}$ are $\beta$-mixing and there exists $c>0$ such that $\beta_n=O(\rme^{-cn})$,
where $\{\beta_n,n\geq 1\}$ is the $\beta$-mixing coefficients sequence; see
\cite[Theorem~3.7]{bradley:2005}. This is a very strong requirement. However, it is satisfied by
many time series models.  We will provide in \Cref{sec:two-models} a fairly general methodology to
check~\Cref{hypo:drift-small}.

\subsection{Convergence of tail arrays sums}
\label{sec:tail-arrays}
Throughout the paper we will write $\bx_{a,b}$ for $(x_a,\dots,x_b)$, $a\leq b\in\Zset$, for any
 sequence $\bx=(x_j)_{j\in\Zset}$. For $q\geq0$, let $\mcl_q$ be the space of measurable functions $\phi$ defined on $\Rset^{h+1}$ such
that
\begin{enumerate}[(i)]
\item there exists a constant $\epsilon>0$ such that
  $|\phi(\vectorbold{x})| \leq \epsilon^{-1} (|\vectorbold{x}|^{q}\vee1) \ind{|\bx|>\epsilon}$
   for $\vectorbold{x}\in \Rset^{h+1} $;
\item for all $j\geq0$, the function $\bx_{0,j+h}\mapsto\phi(\bx_{j,j+h})$ is almost surely
  continuous with respect to $\numultseq{0,j+h}$.
\end{enumerate}
Note that $\mcl_q\subset\mcl_{q'}$ if $q \leq q'$. As an important example, the continuity condition is
satisfied for functions of the form $\bx\mapsto \psi(\bx)\1{(-\infty,\bu]^c}(\bx)$, for
$\bu\in[\vectorbold{0},\vectorbold{\infty})^{h+1}\setminus\{\vectorbold{0}\}$ and a continuous
function $\psi$, because of the homogeneity of the exponent measures.

\begin{lemma}
  \label{lem:cov}
  Under \Cref{hypo:drift-small}, for  $\phi\in\mcl_{q_0/2}$,
  \begin{align}
    \int_{\Rset^{h+1}}
    \phi^2(\vectorbold{x})\numultseq{0,h} (\rmd \vectorbold{x}) +
    \sum_{j=1}^\infty \int_{\Rset^{j+h+1}} |\phi(\vectorbold{x}_{0,h}) \phi(\vectorbold{x}_{j,j+h})|
    \numultseq{0,j+h} (\rmd \vectorbold{x}) < \infty \; .
  \end{align}
\end{lemma}
This lemma will be proved in~\Cref{sec:anticluster-psi}.  Therefore, for $\phi,\phi'\in\mcl_{q_0/2}$, we
can define
\begin{align}
  \sigma^2(\phi) &  = \int_{\Rset^{h+1}} \phi^2(\vectorbold{x})\numultseq{0,h} (\rmd \vectorbold{x})
  + 2 \sum_{j=1}^\infty \int_{\Rset^{j+h+1}} \phi(\vectorbold{x}_{0,h}) \phi(\vectorbold{x}_{j,j+h})
  \numultseq{0,j+h} (\rmd \vectorbold{x}) \; , \label{eq:var-phi} \\
  C(\phi,\phi') & = \frac12\{\sigma^2(\phi+\phi') - \sigma^2(\phi) - \sigma^2(\phi')\} \; .
  \label{eq:covariance-psi}
\end{align}
Let $\TEPmult$ be a Gaussian process indexed by $\mcl_{q_0/2}$ with covariance function $C$. Our
main result is the finite dimensional convergence of the sequence of the processes $\TEPmult_n$
defined in~(\ref{eq:def-tepmult}) and indexed by $\mcl_q$ for $q< q_0/2$. The proof is given in
\Cref{sec:proof-tep-markov}.

\begin{theorem}
  \label{theo:fidi}
  Let \Cref{hypo:drift-small} hold and let $\{u_n\}$ be an increasing sequence such that
  \begin{align}
    \label{eq:un}
    \lim_{n\to\infty} u_n  = \lim_{n\to\infty}n\tail{F}(u_n) = + \infty\;
  \end{align}
  and there exists $\eta>0$ such that
  \begin{align}
    \lim_{n\to\infty} \log^{1+\eta} (n) \bar{F}(u_n) & = 0 \; . \label{eq:condition-un-eta}
  \end{align}
  Assume moreover that either $q=0$ or there exists $\delta>0$ such that $q(2+\delta) \leq q_0$ and
  \begin{align}
    \lim_{n\to\infty} \frac{\log^{1+\eta} (n)} {\{n\bar{F}(u_n)\}^{{\delta/2}}} & = 0 \; . \label{eq:missing-condition}
  \end{align}
    Then $\TEPmult_{n} \fidi \TEPmult$ on $\mcl_q$.
\end{theorem}

\subsubsection*{Comments on the assumptions}
\Cref{theo:fidi} is obtained under \Cref{hypo:drift-small} and the very mild restrictions
(\ref{eq:un}), (\ref{eq:condition-un-eta}) and~(\ref{eq:missing-condition}) on the choice of the
sequence $u_n$. This simplicity is due to the Markovian assumption. In the literature on the fidi
convergence of tail array sums for general mixing time series, assumptions are usually more
involved. We briefly review some of them.

In the first place note that the geometric drift condition implies that the $\beta$-mixing
coefficients decay geometrically fast. This allows to apply the blocking method for the proof with
blocks of size $r_n$ of order $\log^{1+\eta}(n)$. Even though $\beta$-mixing is restrictive, it is a
commonly made assumption and conditions which ensure for $\beta$ mixing are well
known. Geometric ergodicity may also be considered restrictive since it excludes many Markov
  chains, but as will be illustrated in \Cref{sec:counterexample}, non geometrically ergodic Markov
  chains may have a non standard extremal behaviour and in particular may have a vanishing extremal
  index. See \eg\ \cite{roberts:rosenthal:segers:sousa:2006}.

In the literature, the convergence of the variance of the sum within one block is often assumed:
e.g. \cite[Theorem~4.1]{rootzen:leadbetter:dehaan:1998} or \cite{rootzen:2009}.  For the tail
empirical process~(\ref{def:usual-tep}) and other bounded functions which vanish in a neighborhood
of zero, the following condition, introduced by \cite{smith:1992}, has been used:
\begin{align}
  \label{eq:smithconditionS}
  \lim_{m\to\infty} \limsup_{n\to\infty} \sum_{j=m}^{r_n} \pr(X_j>u_ns\mid X_0>u_ns) = 0 \; .
\end{align}
See for instance \cite{drees:segers:warchol:2014} whose Condition C is equivalent
to~(\ref{eq:smithconditionS}).  For sums of unbounded functions (which vanish in a neighborhood of
zero), ad hoc conditions are usually assumed to ensure convergence of the block variance and the
Lindeberg asymptotic negligibility condition for the central limit theorem; see for instance
\cite[Condition~4.2]{leadbetter:rootzen:dehaan:1998} and \cite[Condition~3.15]{drees:rootzen:2010}.
In this paper, we consider an extension of~(\ref{eq:smithconditionS}) to unbounded functions. We
will prove in \Cref{prop:conditiondhs-via-drift} that if $\phi\in\mcl_q$ for $q\leq q_0/2$, then
\Cref{hypo:drift-small} yields
\begin{align}
  \label{eq:smithconditionSextended}
  \lim_{m\to\infty} \limsup_{n\to\infty} \sum_{j=m}^{r_n} \frac{\esp[|\phi((X_0\dots,X_{h})/u_n)|
  |\phi((X_j,\dots,X_{j+h})/u_n)|]}{\bar{F}(u_n)} = 0 \; .
\end{align}
This property in turn will allow to prove convergence of the variance and also some technical
  conditions related to tightness in the functional central limit theorem. Again, this is a
  consequence of geometric ergodicity which may seem to be a high price to pay, but on the other hand the
  condition~(\ref{eq:smithconditionS}) does not hold for the example of \Cref{sec:counterexample}.

\subsection{Statistical applications}
\label{sec:statapp}
For statistical purposes, we will consider the process $\TEPmult_n$ indexed by a class $\mcg$ of
functions and convergence of $\TEPmult_n$ to $\TEPmult$ must be strengthened to weak convergence in
$\ellinfty(\mcg)$, in particular in order to replace the deterministic threshold $u_n$ by an
appropriate sequence of order statistics.  The general theory of weak convergence in
$\ellinfty(\mcg)$ is developed in \cite{vandervaart:wellner:1996} and \cite{gine:nickl:2016} and was
adapted in full generality in the context of cluster statistics in \cite{drees:rootzen:2010}.
We give here an illustration adapted from \cite[Theorem~2.11.1]{vandervaart:wellner:1996}. We
  state it under the simplifying assumption of linear ordering since it is sufficient for the
  forthcoming examples.  More sophisticated examples can be treated as in \cite{drees:rootzen:2010}
  using entropy conditions but are beyond the scope of this paper.

We say that a class $\mcg$ of functions is pointwise separable
(cf. \cite[Section~2.3.3]{vandervaart:wellner:1996}) if there exists a countable subclass
$\mcg_0\subset\mcg$ such that every $g\in\mcg$ is the pointwise limit of a sequence in $\mcg_0$.

Let $\rho_h$ be the pseudometric defined on $\mcl_q$ by
\begin{align*}
  \rho_h^2(\phi,\psi)  = \numultseq{0,h}((\phi-\psi)^2) \; .
\end{align*}
Note that $\rho_h$ is well defined under the assumptions of \Cref{theo:fidi} which imply $q<q_0/2$.

\begin{theorem}
  \label{theo:our-entropy}
  Let the assumptions of \Cref{theo:fidi} hold and let $\mcg\subset\mcl_q$.  Assume moreover that
  \begin{enumerate}[(i)]
  \item\label{item:pointwise-separable}  $\mcg$ is pointwise separable and linearly ordered;
  \item \label{item:envelope-our-entropy} the envelope function
    $\Phi_\mcg=\sup_{\phi\in\mcg} |\phi|$ is in $\mcl_q$;
  \item \label{item:totally-bounded}$(\mcg,\rho_h)$ is totally bounded;
  \item \label{item:continuity} for every sequence $\{\delta_n\}$ which decreases to zero,
    \begin{align}
      \label{eq:continuite-pour-notre-theorem}
      \limsup_{n\to\infty} \sup_{\phi,\psi\in\mcg \atop \rho_h(\phi,\psi)\leq \delta_n}
      \frac{\esp[\{\phi(u_n^{-1}\bX_{0,h})-\psi(u_n^{-1}\bX_{0,h})\}^2]}{\bar{F}(u_n)} = 0  \; .
    \end{align}
  \end{enumerate}
Then $\TEPmult_n\convweak \TEPmult$ in $\ellinfty(\mcg)$.
\end{theorem}
The proof is in \Cref{sec:prooftheofuncmarkov}.

In order to obtain convenient expressions for the limiting variances, we consider the tail process
$\{Y_j,j\in\Zset\}$, introduced in \cite{basrak:segers:2009} and defined as the weak limit (in the
sense of finite dimensional distributions) of the sequence $\{X_j/x,j\in\Zset\}$ given that
$|X_0|>x$, as $x\to\infty$: for $i\leq j \in \Zset$,
\begin{align}
  \pr((Y_{i},\dots,Y_j) \in \cdot) & = \lim_{x\to\infty} \pr(x^{-1}(X_{i},\dots,X_j) \in \cdot
  \mid |X_0|>x) \; .  \label{eq:def-tailprocess}
\end{align}
Then $|Y_0|$ is a Pareto random variable with tail index $\alpha$.  The spectral tail process
$\{\Theta_j,j\in\Zset\}$ is then defined as $\Theta_j=|Y_0|^{-1}Y_j$.

As a first corollary of \Cref{theo:our-entropy}, we obtain functional convergence of the univariate
tail empirical process. Note that contrary to most of the related literature, we do not have
additional conditions to ensure existence of the limiting variance or tightness. Geometric
ergodicity of the underlying Markov chain is sufficient.
\begin{corollary}
  \label{coro:tepusual}
  Let the assumptions of \Cref{theo:fidi} hold. Let $0<s_0<1<t_0$ and assume moreover that
  \begin{align}
    \label{eq:def-tep-bias-uniform}
    \lim_{n\to\infty}\sqrt{n\tail{F}(u_n)}
    & \sup_{s_0\leq s \leq t_0} \left| \frac{\tail{F}(u_n s)}{\tail{F}(u_n)} - s^{-\alpha}\right| \; =0 \; .
  \end{align}
  Then,
  \begin{align}
    \label{eq:convergence-tep}
    \sqrt{n\tail{F}(u_n)}\left\{\frac{1}{n\tail{F}(u_n)}\sum_{j=1}^n\ind{X_j> u_n s}-s^{-\alpha}\right\}\convweak \mathbb{W} \; ,
  \end{align}
  in $\ell_{\infty}([s_0,t_0])$, where $\mathbb{W}$ is a Gaussian process with covariance function
  \begin{align*}
    \cov(\mathbb{W}(s),\mathbb{W}(t))=(s\vee t)^{-\alpha} + \sum_{j=1}^{\infty} \esp[\{(\Theta_j/t)\wedge(1/s)\}_+^\alpha +
    \{(\Theta_j/s)\wedge(1/t)\}_+^\alpha \mid \Theta_0=1]\;.
  \end{align*}
\end{corollary}

Assume that $F$ is continuous and let $k=k(n)$ be an intermediate sequence of integers, that is
$\lim_{n\to\infty} k=\lim_{n\to\infty} n/k=\infty$. Assume that the sequence $u_n$ is such that
$k=n\bar{F}(u_n)$. Let $X_{n;1}\leq \cdots\leq X_{n:n}$ be the increasing order statistics.
Consequently, by Vervaat's lemma (\citet[Proposition~3.3]{resnick:2007}) we obtain that
$X_{n:n-k}/u_n\convprob1$. For statistical applications, the deterministic threshold $u_n$ will be
replaced by $X_{n:n-k}$.  Define the processes $\widehat{\TEPmult}_n$ and $\widehat{\TEPmult}$ on
$\mcl_q$ by
\begin{align*}
  \widehat{\TEPmult}_n(\phi) & =  \sqrt{k} \left\{\frac1{{k}}\sum_{i=1}^n \phi(X_{n:n-k}^{-1}\bX_{j,j+h}) - \numultseq{0,h}(\phi)\right\} \; , \\
  \widehat{\TEPmult}(\phi) & = \TEPmult(\phi) - \numultseq{0,h}(\phi)\TEPmult(\1{(1,\infty)\times\Rset^h}) \; .
\end{align*}
\begin{corollary}
  \label{coro:randomthreshold}
  Let the assumptions of \Cref{theo:fidi} and~(\ref{eq:def-tep-bias-uniform}) hold. Let
  $0<s_0<1<t_0$ and let $\mcg_0\subset\mcl_q$. Define
  $\mcg = \{\phi_s,\phi\in\mcg_0,s\in[s_0,t_0]\}$ with $\phi_s(\bx) = \phi(\bx/s)$. If $\mcg$
  satisfies the assumptions of~\Cref{theo:our-entropy} and
  \begin{align}
    \lim_{n\to\infty} \sqrt{k}  \sup_{s_0\leq s \leq t_0} \sup_{\phi\in\mcg}\left| \frac{\esp[\phi(\bX_{0,h}/(u_ns))]}{\tail{F}(u_n)}
    - s^{-\alpha}\numultseq{0,h}(\phi)\right| = 0 \; , \label{eq:bias-psi}
  \end{align}
  then $ \widehat{\TEPmult}_n \convweak \widehat{\TEPmult}$ on $\ellinfty(\mcg_0)$.
\end{corollary}

\begin{remark}
  The term $\numultseq{0,h}(\phi)\TEPmult(\1{(1,\infty]\times\Rset^h})$ is an effect of the random threshold.
  Conditions~(\ref{eq:def-tep-bias-uniform}) and (\ref{eq:bias-psi}) allow to get rid of the bias
  terms.  These conditions are certainly fulfilled for {some} choices of $k$ but they might be in
  conflict with (\ref{eq:missing-condition}) if the convergence in the definition of regular
  variation is very slow.  However, conditions (\ref{eq:def-tep-bias-uniform}) and
  (\ref{eq:bias-psi}) are generally obtained by means of so-called second order conditions which
  yield polynomial rates of convergence.  We do not pursue in this direction, nor in the very
  important practical issue of a data-driven choice of $k$, both problems being largely beyond the
  scope of this paper.
\end{remark}

The proof of \Cref{coro:tepusual,coro:randomthreshold} is in \Cref{sec:prooftheofeasible}. We now
give several examples.

\begin{example}[{Estimation of the (cluster) large deviation index}]
  \label{sec:cluster}
Assume for simplicity that the random variables $X_j$ are nonnegative.  For $A_h
= \{x_0+\cdots+x_h>1\}$ we consider the following quantity which exists by regular variation:
\begin{align*}
  b_{+,h} = \frac1{h+1}\lim_{x\to\infty} \frac{\pr(X_0+\cdots+X_h>x)}{\tail{F}(x)} =
  \frac{1}{h+1}\numultseq{0,h}(A_h)=\frac{1}{h+1}\int_{A_h}\numult{0,h}(\rmd x) \; .
\end{align*}
If $\{X_j\}$ is a sequence of \iid\ regularly varying nonnegative random variables, then $b_{+,h}=1$
for all $h$. In general, regular variation implies that $(h+1)^{-1} \leq b_{+,h} \leq (h+1)^\alpha$.
It is shown in \cite{mikosch:wintenberger:2014} that the drift condition~(\ref{eq:drift}) implies
that
\begin{align*}
  b_+= \lim_{h\to\infty} b_{+,h} \in [0,\infty) \; .
\end{align*}
The quantity $b_+$ thus defined is called the {cluster index} and is related to the large deviation
behavior of the partial sums $\sum_{j=1}^n X_j$; see \cite{mikosch:wintenberger:2014}.  No
estimators of $b_+$ have been provided yet in the literature. Here we consider an estimator of
$b_{+,h}$. Define
\begin{align*}
  \widehat{b}_{n,+,h} & = \frac1{k(h+1)} \sum_{j=1}^{n} \ind{X_{j}+\cdots+X_{j+h}>\orderstat{n:n-k}} \; .
\end{align*}
With the notation of \Cref{coro:randomthreshold}, $b_{+,h}=\numultseq{0,h}(\phi)$ with
$\phi(\vectorbold{x})=(h+1)^{-1}\ind{x_0+\cdots+x_h>1}$.  The class $\mcg_0$ consists of the
single function $\phi$ and $\mcg=\{\phi_s,s\in[s_0,t_0]\}$. The class $\mcg$ is pointwise separable,
linearly ordered and its envelope function is $\phi_{s_0}$ which belongs to $\mcl_0$. Conditions
\ref{item:totally-bounded} and \ref{item:continuity} of \Cref{theo:our-entropy} are checked in
\Cref{sec:proof-clusterindex}, while \eqref{eq:bias-psi} becomes
\begin{align}
  \lim_{n\to\infty} \sqrt{k} \sup_{s\geq s_0} \left| \frac{\pr(X_0+\cdots+X_h>u_ns)}{\tail{F}(u_n)}
  - s^{-\alpha}\numultseq{0,h}(A_h) \right| = 0 \; .   \label{eq:bias-clusterindex}
\end{align}
Therefore, under the assumptions of \Cref{theo:fidi} with $q=0$ and~(\ref{eq:bias-clusterindex}), we
can apply~\Cref{coro:randomthreshold} and we obtain
\begin{align*}
  \sqrt{k}(\widehat{b}_{n,+,h} - b_{+,h}) \convdistr N(0,\sigma^2_{+,h}) \;  ,
\end{align*}
with
\begin{align*}
  \sigma_{+,h}^2 & = b_{+,h} \{b_{+,h}-1/(h+1)\} + 2 \sum_{j=1}^\infty \left\{
    \sum_{i=0}^h (h+1)^{-\alpha-2} \times \right. \\
  & \phantom{ = }\pr(Y_{-i}\leq1,\dots,Y_{-1}\leq1,  Y_{-i}+\cdots+Y_{h-i}>h+1,Y_{j-i}+\cdots+Y_{j+h-i}>h+1) \\
  &- \frac{b_{+,h}}{h+1} \{ \pr(Y_j+\cdots+Y_{j+h}>1) + \pr(Y_{-j} + \cdots + Y_{-j+h} > 1)\} +
  (b_{+,h})^2 \pr(Y_j>1) \Bigr\} \; .
\end{align*}
\end{example}

\begin{example}[{Conditional tail expectation}]
\label{sec:expected-shortfall}
We assume for simplicity that the time series $\{X_j\}$ is non negative and has tail
index $\alpha>1$.  Then the following limit exists:
\begin{align}
  \label{eq:cte-again}
  \lim_{x\to\infty}\frac{1}{x}\esp[X_h\mid X_0 >x]
  & = \int_{x_0=1}^\infty \int_{\Rset_+^h}  x_h \, \numultseq{0,h}(\rmd \vectorbold{x})
    = {\esp[Y_h] = \frac{\alpha\esp[\Theta_h]}{\alpha-1}} = \CTE_h \; .
\end{align}
Indeed, regular variation implies that the distribution of $x^{-1}(X_0,\dots,X_h)$ conditionally on
$X_0>x$ converges weakly to the probability measure equal to $\numultseq{0,h}$ restricted to
$[1,\infty)\times\Rset_°^h$ and if $\alpha>1$, Potter's bounds ensure that $x^{-1}X_h$ is uniformly
integrable conditionally on $X_0>x$.  Define
\begin{align*}
  \widehat{C}_{n,h} = \frac{1}{k} \sum_{j=1}^{n} \frac{X_{j+h}}{X_{n:n-k}} \ind{X_j>X_{n:n-k}} \; .
\end{align*}
The bias condition~(\ref{eq:bias-psi}) becomes
\begin{align}
  \label{eq:def-tep-bias-linear}
  \lim_{n\to\infty}\sqrt{k} \sup_{s_0\leq s\leq t_0} \left| \frac{\esp\left[X_h\ind{X_0>u_ns}\right]}{u_n\tail{F}(u_n)} - \CTE_hs^{1-\alpha}  \right| =0 \; .
\end{align}
In order to apply \Cref{coro:randomthreshold}, we assume that $\alpha>2$. We set
$\phi(\bx)=x_h\ind{x_0>1}$, $\mcg_0=\{\phi\}$ and $\mcg=\{\phi_s,s\in[s_0,t_0]\}$ with
$0<s_0<1<t_0$. The class $\mcg$ is pointwise separable, linearly ordered and its envelop function is
$s_0^{-1}x_h\ind{x_0>s_0}$ which belongs to $\mcl_2$; it is totally bounded for the metric
$\rho_h$ and Condition~(\ref{eq:continuite-pour-notre-theorem}) holds (see \Cref{sec:proof-shorfall}
for a proof of the latter two points).

Thus we obtain, under the assumptions of \Cref{theo:fidi} with $q=2$,
\begin{align*}
  \sqrt{k} \left\{ \widehat{C}_{n,h} - \CTE_h \right\} \convdistr N(0,\sigma_h^2) \; ,
\end{align*}
with
\begin{align*}
  \sigma_h^2 = \esp[(Y_h&-\CTE_h)^2]  + 2 \sum_{j=1}^\infty \esp[(Y_h-\CTE_h)(Y_{j+h}-\CTE_h)\ind{Y_j>1}] \; .
\end{align*}
\end{example}

\begin{example}[{Estimation of the distribution of the spectral tail process}]
\label{sec:edfstp}

For $h>0$, let $\dfstp_h$ be the distribution function of the spectral tail process $\Theta_h$ at lag $h$,
\begin{align*}
  \dfstp_h(y) = \lim_{x\to\infty} \frac{\pr(|X_0|>x,X_h \leq |X_0|y)}{\pr(|X_0|>x)}\;, \ \ y\in\Rset\;.
\end{align*}
An estimator $\widehat\dfstp_{n,h}$ of $\dfstp_h$ is defined by
\begin{align}
  \widehat\dfstp_{n,h} (y) = \frac1k \sum_{i=1}^n \ind{|X_j|>X_{n:n-k}} \ind{X_{j+h}\leq |X_j|y} \; ,
\end{align}
where $k$ is a non decreasing sequence.  For $y\in\Rset$, define the function $J_y$ on $\Rset^{h+1}$
by
\begin{align*}
  J_{y}(\bx) = \ind{|x_0|>1}\ind{x_h\leq |x_0| y} \; ,
\end{align*}
and set $\limdfstp_h(y)=\TEPmult(J_{y})-L_h(y)\TEPmult(\1{(1,\infty]\times \Rset^h})$.
\end{example}
\begin{theorem}
  \label{theo:fclt-edfstp}
   Let the assumptions of \Cref{theo:fidi} with $q=0$ and (\ref{eq:def-tep-bias-uniform}) hold.
    Assume that the distribution function $\dfstp_h$ of $\Theta_h$ is continuous on $[a,b]$ with
    $a<b\in\Rset$ and  for each $y\in [a,b]$,
  \begin{align}\label{eq:bias-spectral}
    { \lim_{n\to\infty} \sqrt{k} \sup_{s_0\leq s \leq t_0}
    \left| \frac{\pr(|X_0|>u_ns,X_h\leq |X_0|y)}{\bar{F}(u_n)} - s^{-\alpha} \dfstp_h(y)\right|=0 } \; ,
  \end{align}
  where $u_n$ is such that $k=n\bar{F}(u_n)$.  Then
  $\sqrt{k}(\widehat\dfstp_{n,h}-L_h)\fidi\limdfstp_h$ on $[a,b]$.
\end{theorem}
The proof is in \Cref{sec:proof-fclt-edfstp}.  We only consider fidi convergence since a proof of
tightness for a two-parameter process would be much more technical since the linearly ordered
property of the class of functions would fail to hold. Under additional assumptions on the distribution
of the spectral tail process and if the bias condition~(\ref{eq:bias-spectral}) holds uniformly  on $ [a,b]$, tightness with respect to $y$ could be proved by the same techniques
as in \cite[Example~4.4]{drees:rootzen:2010} for a plausibly slower rate of convergence.

Consider a univariate non negative Markov chain which in addition to \Cref{hypo:drift-small}
satisfies $\Theta_{j+1}=A_{j+1} \Theta_j$, $j\geq0$ where $\{A_j,j\geq1\}$ is a sequence of \iid\
random variables with the same distribution as $\Theta_1$. This is the case for most usual Markovian
time series, see \cite{janssen:segers:2014}.  For $h=1$, we can calculate
\begin{align*}
  \var(\limdfstp_1(y)) = \pr(\Theta_1\leq y)\pr(\Theta_1> y) \; .
\end{align*}
This is the same as the variance found in \cite[Corollary~5.2]{drees:segers:warchol:2014} for their
forward estimator which instead of using a random threshold replaces the scaling $k$ by the random
number of exceedances of the Markov chain $\{X_j\}$ above the  deterministic threshold $u_n$.

\subsection{Checking \Cref{hypo:drift-small}}
\label{sec:two-models}
We now show how to check \Cref{hypo:drift-small}. Irreducibility and the drift
condition~(\ref{eq:drift}) are well known for most Markovian time series models, but the link with
the conditions of extreme value theory has not been yet fully
investigated. \cite{janssen:segers:2014} used the functional autoregressive representation of most
Markov chains to obtain conditions for the whole time series to be regularly varying when the
stationary distribution is regularly varying. We build on this approach to check
\Cref{hypo:drift-small}.  Assume that $\{\mby_j,j\in\Nset\}$ is a $\Rset^d$-valued Markov chain which
admits a functional autoregressive representation
\begin{align}
  \label{eq:FAR}
  \mby_{j+1} = \Phi(\mby_j,Z_{j+1}) \; , \ \ j\geq0 \; ,
\end{align}
where $\{Z,Z_j,j\geq1\}$ is a sequence of \iid\ random variables with values in a measurable space
$\Eset$ and $\Phi: \Rset^d \times \Eset\to \Rset^d$ is a measurable map.  Fix a norm $\|\cdot\|$ and
denote the unit sphere by $\sphere{d-1}$. We assume that there exist $q_0<\alpha$, $ \zeta_1 , \zeta_2>0$ and a map
$V:\Rset^d\to[1,\infty)$ such that
\begin{subequations}
  \label{eq:other-conditions}
  \begin{align}
    \label{eq:V-normliike}
    & \zeta_1 (\|\bx\|\vee1)^{q_0} \leq V(\bx) \leq \zeta_2 (\|\bx\|\vee1)^{q_0} \; , \\
    & \sup_{\bx\in\Rset^d} \|\bx\|^{-q_0}\esp[\|\Phi(\bx,Z) \|^{q_0}] <\infty \; ,  \label{eq:sinul}\\
    \label{eq:asymptotic-contractivity}
    &  \limsup_{\|\bx\|\to\infty}  \frac{\esp\left[V(\Phi(\bx,Z)) \right]}{V(\bx)} < 1 \; .
  \end{align}
\end{subequations}
Under these conditions,~(\ref{eq:drift}) holds. Indeed, under (\ref{eq:asymptotic-contractivity}),
we can choose $\gamma\in(0,1)$ such that for $r$ sufficiently large
\begin{align*}
  \sup_{\|\bx\|>r} \frac{\esp[V(\Phi(\bx,Z))]}{V(\bx)} < \gamma \; .
\end{align*}
 The upper bounds in~(\ref{eq:V-normliike}) and~(\ref{eq:sinul}) ensure that
  $\esp[V(\Phi(\bx,Z))]$ is bounded on compact sets, and the lower bound in~(\ref{eq:V-normliike})
  ensures that $V$ is unbounded outside compact sets; thus (\ref{eq:drift}) holds.

In most examples, $\{\mby_j\}$ is itself a regularly varying time series and $g$ is a homogeneous
function, so that $\{X_j\}$ is regularly varying and (\ref{eq:def-Q}) holds by the conditions
$|g|^{q_0}\leq cV$ and~(\ref{eq:V-normliike}).
We now consider two examples.

\subsubsection*{AR(p) with regularly varying innovations}
%\label{sec:arp}

Convergence of the tail empirical processes of exceedances for infinite order moving averages has
been obtained in the case of finite variance innovation; for infinite variance innovations it was
proved only in the case of an AR(1) process in \cite{drees:2003}. We next show that
\Cref{hypo:drift-small} holds for general causal invertible AR($p$) models.
\begin{corollary}
  \label{cor:arp}
  Assume that $\sequenceshort{X}{j}$ is an AR($p$) time series
  \begin{align*}
    X_j=\varphi_1X_{j-1}+\cdots+\varphi_pX_{j-p}+\varepsilon_j\;,\ \ j\geq 1\;,
  \end{align*}
  that satisfies the following conditions:
  \begin{itemize}
  \item $\sequenceshort{\varepsilon}{j}$ is a sequence of \iid\ random variables, regularly varying
    with index $\alpha$ whose common density possesses an absolutely continuous component;
  \item the spectral radius of the matrix
    \begin{align*}
      A = \left(
        \begin{array}{ccccc}
          \varphi_1 & \varphi_2 &\varphi_3 &  \cdots &\varphi_p\\
          1 & 0 & 0 & \cdots & 0\\
          0 & 1 & 0 & \cdots & 0\\
          \vdots & \vdots &\vdots & \ddots & \vdots\\
          0 & 0 & \cdots & 1 & 0
        \end{array}
      \right)\;
    \end{align*}
    is smaller than 1.
  \item if $\alpha\leq 2$, then $\sum_{i=1}^p|\varphi_i|^r<1$ for $r=\min\{1,\alpha\}$.
  \end{itemize}
  Let $\{\mby_j,j\geq 1\}$ be the $\Rset^p$-valued vector-autoregressive Markov chain
  \begin{align}
    \label{eq:var}
    \mby_j=A \mby_{j-1}+Z_j
  \end{align}
  with
  \begin{align*}
    \mby_j=(X_j,\ldots,X_{j-p+1})^T\;, \ \ \ \mbz_j=(\varepsilon_j,0,\ldots,0)^T\;.
  \end{align*}
  Then there exists a norm $\|\cdot\|$ on $\Rset^p$ such that \Cref{hypo:drift-small} holds with
  $V(\bx) = 1+\|\bx\|^{q_0}$ for any $q_0<\alpha$.
\end{corollary}
\begin{proof}
  Under the stated assumptions, the Markov chain $\{\mby_j,j\geq1\}$ is positive Harris recurrent on
  $\Rset^p$, all compact sets are small sets, see \cite[Example 2.6 (d)]{alsmeyer:2003}; the
  stationary distribution is given by $\mby_j = \sum_{k=0}^\infty A^kZ_{j-k}$ and is regularly
  varying, see \cite{hult:samorodnitsky:2008}.  The $AR(p)$ process admits the
  representation~(\ref{eq:FAR}) with $\Phi$ given by
  \begin{align*}
    \Phi(\bx,\bz) = A\bx+\bz \; .
  \end{align*}
  Therefore (\ref{eq:V-normliike}) and~(\ref{eq:sinul}) hold for $V=1+\|\cdot\|^{q_0}$ for any norm
  $\|\cdot\|$ on $\Rset^p$ with $\limitmap(\bx,\bz)=A\bx$ and for any $q_0$ such that
  $\esp[|\varepsilon_0|^{q_0}]<\infty$.  We must show that there exists a norm such that
  condition~(\ref{eq:asymptotic-contractivity}) is fulfilled. Let $\lambda$ be the spectral radius
  of the matrix $A$. Fix $\epsilon$ such that $\gamma=\lambda+\epsilon<1$. Then there exists a
  norm (depending on $A$) such that
 \begin{align*}
   \sup_{\bx\in\Rset^p\atop \|\bx\|_p=1} \|A \bx\| \leq \gamma \; ;
 \end{align*}
 see e.g.  \cite[Proposition~4.24 and Example~6.35]{douc:moulines:stoffer:2014}.  This
 yields~(\ref{eq:asymptotic-contractivity}) with $V(\bx) = 1+\|\bx\|^{q_0}$.
\end{proof}

\subsubsection*{Threshold ARCH}
\label{sec:tarch}

We consider the Threshold-ARCH model. It was proved to have a regularly varying stationary
distribution by \cite{cline:2007}. We show here that it satisfies \Cref{hypo:drift-small}.
\begin{corollary}
  \label{cor:tarch}
  Let $\xi\in\Rset$. Assume that $\{X_j\}$ follows a Threshold-ARCH model,
  \begin{align}
    \label{eq:tarch}
    X_j= (b_{10}+b_{11} X_{j-1}^2)^{1/2}Z_j\ind{X_{j-1}<\xi} +  (b_{20}+b_{21} X_{j-1}^2)^{1/2}Z_j\ind{X_{j-1}\geq\xi}\;,
  \end{align}
  that satisfies the following conditions:
  \begin{itemize}
  \item $b_{10},b_{11},b_{20},b_{21}> 0$;
  \item $\{Z_j,j\in\Zset\}$ is a sequence of \iid\ random variables such that
    $\esp[|Z_1|^{\beta}]<\infty$ for all $\beta>0$;
  \item the distribution of $Z_1$ has a bounded density with respect to Lebesgue's measure not
    vanishing in a neighbourhood of zero;
  \item $\esp[\log\{|Z_1|(\sqrt{b_{11}}\ind{Z_1<0}+\sqrt{b_{21}}\ind{Z_1\geq0})\}]<0$.
  \end{itemize}
  Then the Markov chain $\{X_j\}$ is an irreducible and aperiodic chain;
  its stationary distribution is regularly varying with index $\alpha$
  obtained by solving
  \begin{align}
   \label{eq:tail-tarch}
    b_{11}^{\alpha/2} \esp[|Z_1|^{\alpha} \ind{Z_1<0}] + b_{21}^{\alpha/2} \esp[|Z_1|^{\alpha} \ind{Z_1\geq 0}] = 1 \; .
  \end{align}
  \Cref{hypo:drift-small} holds with
  $V(x) = 1+|x|^{q_0}(b_{11}^{q_0/2}\ind{x<0}+b_{21}^{q_0/2}\ind{x\geq0})$ for any $q_0<\alpha$.
\end{corollary}

\begin{proof}
  The statements about ergodicity and the tail of the marginal distribution are proved in
  \cite{cline:2007}. The chain has the representation (\ref{eq:FAR}) with
  \begin{align*}
    \Phi(x,z) =  (b_{10}+b_{11} x^2)^{1/2}z\ind{x<\xi} +  (b_{20}+b_{21} x^2)^{1/2}z\ind{x\geq\xi} \; .
  \end{align*}
  Thus (\ref{eq:V-normliike}) and (\ref{eq:sinul}) hold for all $q_0>0$.  For $q_0<\alpha$, set
  $\lambda_{q_0} =
  b_{11}^{q_0/2}\esp[|Z_1|^{q_0}\ind{Z_1<0}]+b_{21}^{q_0/2}\esp[|Z_1|^{q_0}\ind{Z_1\geq0}]$.
  Then (\ref{eq:tail-tarch}) guarantees that $\lambda_{q_0}<1$ and it is readily checked that
  \begin{align*}
    \lim_{|x|\to\infty} \frac{\esp[V(\Phi(x,Z_1))]}{V(x)} = \lambda_{q_0} <1\; .
  \end{align*}
  This proves~(\ref{eq:asymptotic-contractivity}).
\end{proof}

\subsection{A counterexample}
\label{sec:counterexample}
Considering the sum of an \iid\ regularly varying sequence and a non geometrically ergodic lighter
tailed Markov chain, one can easily see that the geometric drift condition is not a necessary
condition for the results of \Cref{sec:tail-arrays,sec:statapp} to hold. However, when the geometric
drift condition does not hold, it is easy to build counterexamples of non geometrically ergodic
Markov chains which exhibit a highly non standard behaviour of their tail empirical process. In
particular, their extremal index is 0. We now provide a toy example of such a non standard
behaviour.

Let $\sequence{Z}{j}{\Zset}$ be a sequence of \iid\ positive integer valued random variables with regularly
varying right tail with index $\tailz>1$. Define the Markov chain $\{X_j,j\geq0\}$ by the following
recursion:
\begin{align*}
  X_j =
  \begin{cases}
    X_{j-1} -1 & \mbox{ if } X_{j-1} > 1 \; , \\
    Z_j &  \mbox{ if } X_{j-1} = 1  \; .
  \end{cases}
\end{align*}
Since $\tailz>1$, the chain admits a stationary distribution $\pi$ on $\Nset$ given by
\begin{align*}
  \pi(n) = \frac{\pr(Z_0\geq n)}{\esp[Z_0]} \; ,\ \  n\geq 1 \; .
\end{align*}
To avoid confusion, we will denote the distributions functions of $Z_0$ and $X_0$ (when the initial
distribution is $\pi$) by $F_Z$ and $F_X$, respectively.  The tail $\tail{F}_X$ of the stationary
distribution is then regularly varying with index $\alpha=\tailz-1$, since it is given by
\begin{align}\label{eq:X-Z}
  \tail{F}_X(x) = \frac{\esp[(Z_0-[x])_+]}{\esp[Z_0]} \sim \frac{x\tail{F}_Z(x) }{ \tailz \esp[Z_0]}\;, \ \ x\to\infty \; .
\end{align}
We assume for simplicity that $\pr(Z_0=n)>0$ for all $n\geq1$; this implies that the chain is
irreducible and aperiodic and the state $\{1\}$ is a recurrent atom. The distribution of the return time
$\return{1}$ to the atom $\{1\}$, when the chains start from $\{1\}$, is the distribution of $Z_0$. Hence the
chain is not geometrically ergodic since under the assumption on $Z_0$,
$\esp_1[\kappa^{\tau_1}]=\esp[\kappa^{Z_0}]=\infty$ for all $\kappa>1$. Moreover, the extremal index
of the chain is 0, by an application of \cite[Theorem~3.2 and Eq.~(4.2)]{rootzen:1988}.

Let $\{u_n\}$ be a scaling sequence. Consider the tail empirical distribution function defined in
\eqref{def:usual-tep} and $T_n(s) = \esp[\widetilde{T}_n(s)] = \tail{F}_X(u_ns)/\tail{F}_{X}(u_n)$.
Let $\{a_n\}$ be a scaling sequence such that $\lim_{n\to\infty} n\pr(Z_0>a_n)=1$.
\begin{proposition}
  \label{prop:counterexample}
  \begin{itemize}
  \item If $\lim_{n\to\infty}n\tail{F}_Z(u_n) = 0$, then
    $\lim_{n\to\infty}\pr(\widetilde{T}_n(s)\ne0)=0$.
  \item If $\tailz\in(1,2)$ and $\lim_{n\to\infty}n\tail{F}_Z(u_n) = \infty$, then there exists a
    $\tailz$-stable random variable~$\Lambda$ such that for every $s>0$,
    $a_n^{-1}n\tail{F}_X(u_n) \{\widetilde{T}_n(s)-T_n(s)\}\convdistr\Lambda$.
  \item If $\tailz>2$, $\lim_{n\to\infty}n\tail{F}_Z(u_n)= \infty$ and $s_0>0$, then the process
    $s\to \sqrt{n \tail{F}_Z(u_n)}\{\widetilde{T}_n(s)-T_n(s)\}$ converges weakly in
    $\mathbb{D}([s_0,\infty))$ equipped with the Skorokhod $J_1$-topology to a centered Gaussian
    process $\widetilde\TEP$ with covariance function
    \begin{align*}
      C(s,t)=\frac{(\tailz+1) t^{1-\tailz}}{ \tailz (\tailz-1)}-\frac{ s t^{-\tailz}}{\tailz},\qquad
      s<t\,.
    \end{align*}
  \end{itemize}
\end{proposition}

\begin{remark}
  \label{rem:counterexample}
  In the standard situation (for example, under the geometric drift condition), a non degenerate
  limit is expected if $n\tail{F}_X(u_n)\to\infty$. Since $\tail{F}_X(u_n)\sim u_n\tail{F}_Z(u_n)$,
  it may happen simultaneously that $n\tail{F}_X(u_n)\to\infty$ and $n\tail{F}_Z(u_n)\to0$. The
  appropriate threshold is determined by the distribution of $Z_0$ and not by the stationary
  distribution of the chain.
\end{remark}

\section{Central limit theorem for tail array sums}
\label{sec:general}
In this section we prove the central limit theorem for tail array sums in the general framework of a
strictly stationary regularly varying $\beta$-mixing sequence with values in $\Rset^d$.

\subsection{\Anticlustering\ condition}
\label{sec:anticluster-psi}

The main tool to prove our results is a modified form of the condition~\eqref{eq:smithconditionS}.
Precisely, let $\{\xi_j,j\in\Zset\}$ be a regularly varying stationary $\Rset^d$-valued
sequence. Let $|\cdot|$ be an arbitrary norm on $\Rset^d$ and let $\HH$ be the distribution function
of $|\xi_0|$.
\begin{hypothesis}[Condition~\ref{eq:anticlustering-weighted}]
  \label{hypo:anticlustering-psi-general}
  Let $\{u_n\}$ and $\{r_n\}$ be sequences which tend to infinity, $\{r_n\}$ being integer valued, and $\psi:\Rset^d\to\Rset$ be a
  function which vanishes in a neighborhood of zero (\ie\ there exists $\epsilon>0$ such that
  $\psi(\vectorbold{x})=0$ if $|\vectorbold{x}|\leq \epsilon$). For all $s,t>0$,
  \begin{align}
    \label{eq:anticlustering-weighted}
    \lim_{L\to\infty} \limsup_{n\to\infty} \frac{1}{\tail{\HH}(u_n)} \sum_{j= L+1}^{r_n} \esp
    \left[ |\psi({s\xi_0}/{u_n})| | \psi (t\xi_j/{u_n}) | \right] = 0 \;
    . \tag{$\mathcal{S}(u_n,r_n,\psi)$}
  \end{align}
\end{hypothesis}
Note that by stationarity,
\begin{align*}
  \sum_{L < -j \leq r_n} \esp \left[ |\psi({s\xi_0}/{u_n})| | \psi (t\xi_j/{u_n}) | \right]
  & = \sum_{L < j \leq r_n} \esp \left[ |\psi({s\xi_j}/{u_n})| | \psi (t\xi_{0}/{u_n}) | \right] \;  .
\end{align*}
Therefore Condition~\ref{eq:anticlustering-weighted} can be equivalently written as a one-sided or a
two sided sum. Note also that for bounded $\psi$ Condition~\ref{eq:anticlustering-weighted} is
implied by \eqref{eq:smithconditionS} applied to $\{\xi_j\}$.  An equivalent formulation of
\eqref{eq:smithconditionS} is condition (C) of \cite{drees:segers:warchol:2014}.

By assumption, for $j\geq0$ the vector $(\xi_0,\dots,\xi_j)$ is regularly varying so we can define
the exponent measure $\nuxi_{0j}$ of $(\xi_0,\dots,\xi_j)$, that is the Radon measure on
$\Rset^{d(j+1)}\setminus\{\vectorbold{0}\}$ such that
\begin{align*}
  \lim_{x\to\infty} \frac{\pr((\xi_0,\dots,\xi_j) \in x A)}{\pr(|\xi_0|>x)} = \nuxi_{0j}(A) \;
\end{align*}
for relatively compact sets $A$ in $\Rset^{d(j+1)}\setminus\{\vectorbold{0}\}$ such that
$\mu_{0j}(\partial A)=0$.  For $j=0$ we simply write $\mu_0$ for  $\mu_{00}$.  Note that if $d=h+1$ and
$\xi_0=(X_0,\dots,X_h)$, then $\nuxi_0$ is proportional to the measure $\numultseq{0,h}$ appearing
in~(\ref{eq:rv-positive}).  For measurable functions~$\phi,\phi'$, define formally
\begin{multline}
  \Gamma(\phi,\phi')   = \int_{\Rset^d} \phi(\bx_0)\phi'(\bx_0)\nuxi_0(\rmd\vectorbold{x}_0) \\
  + \sum_{j=1}^\infty\int_{\Rset^{d(j+1)}} \{\phi(\bx_0) \phi'(\bx_j)+\phi(\bx_0) \phi'(\bx_j)\}
  \nuxi_{0j}(\rmd\bx) \; , \label{eq:def-sigma2psi-general}
\end{multline}
with $\vectorbold{x}=(\bx_0,\dots,\bx_j)\in\Rset^{d(j+1)}$.
In order to provide conditions for the series in (\ref{eq:def-sigma2psi-general}) to be summable, we
define the following set of functions.
\begin{definition}
  Let $\psi$ be a non negative function defined on $\Rset^d$.  The space $\mcm_\psi$ is the
  set of measurable functions $\phi$ defined on $\Rset^d$ having the following properties:
\begin{itemize}
\item $|\phi|\leq \constant \cdot \psi$, where $\constant$ depends on  $\phi$;
\item for all $j\geq0$, the function defined on $\Rset^{d(j+1)}$ by $(\bx_{0},\dots,\bx_j)\to\phi(\bx_j)$
  is almost surely continuous \wrt\ $\nuxi_{0j}$.
\end{itemize}
\end{definition}
Obviously, $\mcm_\psi$ is a linear space.

\begin{lemma}
  \label{lem:covar-existence}
  Let $\sequence{\xi}{j}{\Zset}$ be a strictly stationary regularly varying sequence and let $\HH$
  be the distribution function of $|\xi_0|$. Let $\{u_n\}$ and $\{r_n\}$ be non decreasing sequences
  tending to infinity, $\{r_n\}$ being integer valued.  Let $\psi$ be a non negative measurable
  function which vanishes in a neighborhood of zero, satisfying
  Condition~\ref{eq:anticlustering-weighted} and for which there exists $\delta>0$ such that
  \begin{align}
    \label{eq:psi-second-moment-1}
    \sup_{n\geq1} & \frac{ \esp\left[ \psi^{2+\delta}(\xi_0/u_n)\right]}{\tail{\HH}(u_n)} < \infty \; .
  \end{align}
  Then for all $\phi,\phi'\in\mcm_\psi$, the series defining $\Gamma(\phi,\phi')$
  in~(\ref{eq:def-sigma2psi-general}) is absolutely summable and
  \begin{align}
    \Gamma(\phi,\phi')
    & = \lim_{n\to\infty} \frac1{r_n\bar{\HH}(u_n)} \esp\left[ \left( \sum_{j=1}^{r_n}
      \phi(\xi_j/u_n)\right) \left( \sum_{j=1}^{r_n} \phi'(\xi_j/u_n)\right) \right]   \label{eq:lim-squared-psi}  \\
    &  = \lim_{L\to\infty} \lim_{n\to\infty}
      \sum_{j=-L}^L \frac{\esp[\phi(\xi_0/u_n) \phi'(\xi_{j}/u_n)]} {\tail{\HH}(u_n)}  \;
      . \label{eq:var-psi-alt}
  \end{align}
  If moreover
  \begin{align}
    \lim_{n\to\infty} r_n\bar{\HH}(u_n) = 0 \; , \label{eq:rate-fidi1}
  \end{align}
  then
  \begin{align}
    \Gamma(\phi,\phi')= \lim_{n\to\infty} \frac{ 1}{r_n\tail{\HH}(u_n)}\cov\left(\sum_{j=1}^{r_n}
      \phi(\xi_{j}/u_n),\sum_{j=1}^{r_n} \phi'(\xi_{j}/u_n)\right) \; .  \label{eq:var-psi-sigma}
  \end{align}

\end{lemma}

Throughout this section, we will use the following notation. Set
\begin{align*}
  \phi_{n,j} & = \phi(\xi_{j}/u_n) \; ,\ \ c_j(\phi) = \int_{\Rset^{d(j+1)}} \phi(\vectorbold{x}_{0}) \phi(\vectorbold{x}_{j})  \nuxi_{0j} (\rmd \vectorbold{x})\;, \ \ j\geq 0 \; .
\end{align*}

\begin{proof}
  Since $\mcm_\psi$ is a linear space and by the identity $2xy = (x+y)^2-x^2-y^2$, it suffices to
  prove these identities for $\phi=\phi'$. We restrict ourselves to non negative functions $\phi$, the extension for general $\phi\in \mcm_\psi$ being straightforward. We must first prove that
  condition~(\ref{eq:psi-second-moment-1}) ensures that the coefficients $c_j(\phi)$ are well
  defined.  Since the support of $\phi$ is bounded away from zero and  vague convergence
  coincides with weak convergence on such sets, if $\phi$ is bounded, then the continuous mapping
  theorem yields
  \begin{align}
    \lim_{n\to\infty} \frac{\esp[\phi_{n,0}\phi_{n,j}]}{\bar{\HH}(u_n)} = c_j(\phi) \; .  \label{eq:tralala}
  \end{align}
  If $\phi$ is unbounded and condition (\ref{eq:psi-second-moment-1}) holds, then for all $A>0$,
  applying Markov and H\"older inequalities, we obtain
  \begin{align*}
    \lim_{A\to\infty}    \limsup_{n\to\infty} \frac{\esp[\phi_{n,0}\phi_{n,j}\ind{|\phi_{n,0}\phi_{n,j}|>A}]}{\tail{\HH}(u_n)}
    \leq \lim_{A\to\infty}   \constant \cdot A^{-\delta/2} \sup_{n\geq1} \frac{\esp[\psi^{2+\delta}(\xi_0/u_n)]}{\tail{\HH}(u_n)} = 0 \; .
  \end{align*}
  This allows to use a truncation argument and prove that (\ref{eq:tralala}) holds.

  We now prove that the series $\Gamma(\phi,\phi)$ is summable and without loss of generality we assume
  that $\phi$ is nonnegative.  Fix $\eta>0$. Applying~\ref{eq:anticlustering-weighted} and the
  fact that $\phi$ is bounded by a multiple of  $\psi$, we can choose $L$ such that, for every $R\geq L$
  \begin{align*}
    \lim_{n\to\infty} \sum_{j=L}^R \frac{\esp[\phi_{n,0}\phi_{n,j}]}{\tail{\HH}(u_n)} \leq \eta \;    .
  \end{align*}
  This yields that for every $\eta>0$, large enough $L$ and all $R\geq L$, $\sum_{j=L}^{R}
  c_j(\phi) \leq \eta$ and this means that the series $\sum_{j=1}^\infty
  c_j(\phi)$ is summable and that (\ref{eq:var-psi-alt}) holds. To prove (\ref{eq:lim-squared-psi}), write
  \begin{align*}
    \frac{ \esp \left[ \left( \sum_{j=1}^{r_n} \phi_{n,j}\right)^2 \right]}{r_n\bar{\HH}(u_n)} =
    \frac{\esp[\phi_{n,0}^2]}{\bar{\HH}(u_n)}
    &   + 2 \sum_{j=1}^L (1-j/r_n)\frac{\esp[\phi_{n,0}\phi_{n,j}]}{\tail{\HH}(u_n)}
     + 2\sum_{j=L+1}^{r_n} (1-j/r_n) \frac{\esp[\phi_{n,0} \phi_{n,j}]}{\tail{\HH}(u_n)} \; .
  \end{align*}
  By Condition~\ref{eq:anticlustering-weighted}, for every $\eta>0$, we can choose $L$ in such a
  way that  the last term above is less than $\eta$. This yields
  \begin{align*}
    \limsup_{n\to\infty} \left| \sum_{1 \leq j,j' \leq r_n}
    \frac{\esp[\phi_{n,j}\phi_{n,j'}]}{r_n\tail{\HH}(u_n)} -  c_0(\phi) -
    2\sum_{j=1}^L c_j(\phi) \right| \leq \eta    \; .
  \end{align*}
  Since the series $\sum_{j=1}^L c_j(\phi) $ is convergent we can also choose $L$ in such a way that
  $\sum_{j=L+1}^\infty c_j(\phi) \leq \eta$. This yields
  \begin{align*}
    \limsup_{n\to\infty} \left| \sum_{1 \leq j,j' \leq r_n} \frac{\esp[\phi_{n,j}\phi_{n,j'}]}
      {r_n\tail{\HH}(u_n)} - \Gamma(\phi,\phi) \right| \leq 3 \eta \; .
  \end{align*}
  Since $\eta$ is arbitrary, this proves~(\ref{eq:lim-squared-psi}).  Finally, with
  \begin{align*}
    S_n(\phi)=\sum_{j=1}^{r_n} \phi_{n,j} \; ,
  \end{align*}
  we have
  \begin{align*}
    \frac{ \var(S_n(\phi))}{r_n\tail{\HH}(u_n)} & = \frac{\esp[S_n^2(\phi)]}{r_n\tail{\HH}(u_n)} -
    \frac{r_n^2(\esp[\phi_{n,0}])^2} {r_n\tail{\HH}(u_n)} = \frac{\esp[S_n^2(\phi)]}{r_n\tail{\HH}(u_n)}
    + O(r_n\bar{\HH}(u_n)) \; .
  \end{align*}
  Under condition~(\ref{eq:rate-fidi1}), the last term is $o(1)$. This
  proves~(\ref{eq:var-psi-sigma}).
\end{proof}

\subsection{Fidi convergence of tail array sums}
In this section, we prove a theorem on convergence of tail array sums which complements the results
of \cite{rootzen:leadbetter:dehaan:1998} and \cite[Theorem~2.3]{drees:rootzen:2010}. We show that
under $\beta$-mixing, condition \ref{eq:anticlustering-weighted} is the main ingredient of the proof
of the central limit theorem.

Define the process $\mathbb{W}_n$ on $\mcm_\psi$ by
\begin{align*}
  \mathbb{W}_n(\phi)
  = \frac1{\sqrt{n\bar{\HH}(u_n)}} \sum_{j=1}^n \{\phi(\xi_{j}/u_n) -\esp[\phi(\xi_{0}/u_n)]\} \; , \ \ \phi\in\mcm_\psi \; .
\end{align*}
\begin{theorem}
  \label{theo:tail-array-fidi}
  Let $\sequence{\xi}{j}{\Zset}$ be a strictly stationary regularly varying sequence and let~$\HH$
  be the distribution function of $|\xi_0|$. Let $\{u_n\}$ and $\{r_n\}$ be non decreasing sequences
  which tend to infinity, $\{r_n\}$ being integer valued, such that
  \begin{align}
    \label{eq:rate-condition-fidi}
    \lim_{n\to\infty}u_n= \lim_{n\to\infty} n\bar{\HH}(u_n) = \infty \; , \ \ \lim_{n\to\infty} r_n\bar{\HH}(u_n) =  0   \; .
  \end{align}
  Let $\psi$ be a function which vanishes in a neighborhood of zero and such that
  \ref{eq:anticlustering-weighted} holds. Assume that either $\psi$ is bounded or there exists
  $\delta\in (0,1]$ such that~(\ref{eq:psi-second-moment-1}) holds and
  \begin{align}
    \lim_{n\to\infty} \frac{r_n }{\left(n\bar{\HH}(u_n)\right)^{{\delta/2}}} = 0 \; .
     \label{eq:rate-condition-fidi-unbounded}
  \end{align}
  Assume that the sequence $\{\xi_j,j\in\Zset\}$ is $\beta$-mixing with coefficients
  $\{\beta_n,n\geq1\}$ and there exists a sequence $\{\ell_n\}$ such that
  \begin{align}
    \label{eq:beta-fidi}
    \ell_n\to\infty \; , \ \ell_n/r_n\to 0 \; , \ \lim_{n\to\infty} n
    \beta_{\ell_n}/r_n = 0 \; .
  \end{align}
  Let $\mathbb{W}$ be a Gaussian process indexed by $\mcm_\psi$ with covariance function $\Gamma$
  defined in~(\ref{eq:def-sigma2psi-general}).  Then $\mathbb{W}_n \fidi \mathbb{W}$ on
  $\mcm_{\psi}$.
\end{theorem}

\begin{remark}
  It is possible to find sequences $\ell_n$ and $r_n$ that satisfy (\ref{eq:beta-fidi}) if the
  $\beta$-mixing coefficients $\beta_n$ satisfy $\beta_n=O(n^{-a})$ for some $a>0$. A suitable
  choice is then $r_n=n^\zeta$ and $\ell_n=n^\eta$ with $0<\eta<\zeta<1$ and $\zeta+a\eta>1$.
\end{remark}

\begin{proof}[Proof of \Cref{theo:tail-array-fidi}]
  {Since $\mcm_\psi$ is a linear space and $\Gamma$ is a quadratic form, it suffices to prove
    the central limit theorem for an arbitrary $\phi\in\mcm_\psi$.} For $i=1,\ldots,[n/r_n]$, define
  \begin{align}
    \label{eq:def-Sn}
    S_{n,i}(\phi) = \sum_{j=(i-1)r_n+1}^{ir_n} \phi(\xi_{j}/u_n) \; , \ \ \bar{S}_{n,i}(\phi) =
    S_{n,i}(\phi)-\esp[S_{n,i}(\phi)] \; .
  \end{align}
  Arguing as in \cite[Lemma~5.1]{drees:rootzen:2010}, Condition (\ref{eq:beta-fidi}) implies that it
  suffices to prove the central limit theorem for the sum with independent blocks of
    length~$r_n$ having the same marginal distribution as the original blocks
  $(\xi_{(i-1)r_n+1},\dots,\xi_{ir_n}$) and that we can remove a smaller block $(\xi_{ir_n-\ell_n+1},\dots,\xi_{ir_n}$) of size
  $\ell_n$ at the end of each large block. To this end, we must prove the convergence of the
  variance and the Lindeberg asymptotic negligibility condition on the sum of independent big blocks.
  By \Cref{lem:covar-existence}, we already know that
  \begin{align*}
    \lim_{n\to\infty} \frac{\var(S_{n,1}(\phi))}{r_n\bar{\HH}(u_n)} =
    \lim_{n\to\infty}   \frac{\var\left(\sum_{j=1}^{r_n} \phi(\xi_j/u_n)\right)}{r_n\bar{\HH}(u_n)} =  \Gamma(\phi,\phi) \; .
  \end{align*}
  Since $\ell_n\leq r_n$, $\mathcal{S}(u_n,r_n,\psi)$ implies $\mathcal{S}(u_n,\ell_n,\psi)$ and
  hence the limit above holds with $r_n$ replaced with $\ell_n$.  By $\ell_n/r_n\to 0$, this also entails that
  \begin{align*}
    \lim_{n\to\infty} \frac1{r_n\bar{\HH}(u_n)} \var\left(\sum_{j=1}^{\ell_n} \phi(\xi_j/u_n)\right) =  0 \; .
  \end{align*}
  This means that the small blocks do not contribute to the limit.  Therefore, we only need to prove
  the asymptotic negligibility condition. This is done  in \Cref{lem:asympt-neglig-bounded} in
  the bounded case and \Cref{lem:asympt-neglig-unbounded} in the unbounded case.
\end{proof}

\begin{lemma}
  \label{lem:asympt-neglig-bounded}
  Let $\psi$ be a bounded non negative function which vanishes in a neighborhood of zero and such
  that \ref{eq:anticlustering-weighted} holds. If (\ref{eq:rate-condition-fidi}) holds, then for all
  $\eta>0$ and all $\phi\in\mcm_\psi$,
  \begin{multline*}
    \lim_{n\to0} \frac1{r_n\bar{\HH}(u_n)}
     \esp \left[ {S}_{n,1}^2(\phi) \ind{|{S}_{n,1}(\phi)|>\eta\sqrt{n\bar{\HH}(u_n)}} \right]  \\
     = \lim_{n\to0} \frac1{r_n\bar{\HH}(u_n)} \esp \left[ \bar{S}_{n,1}^2(\phi)
      \ind{|\bar{S}_{n,1}(\phi)|>\eta\sqrt{n\bar{\HH}(u_n)}} \right] = 0 \; .
  \end{multline*}
\end{lemma}

\begin{proof}
  Write for brevity $v_n=\sqrt{n\bar{\HH}(u_n)}$, $S_n(\phi)$ for $S_{n,1}(\phi)$ and
  $\bar{S}_n(\phi)$ for $\bar{S}_{n,1}(\phi)$.  At the first step we note that the centering can be
  omitted.  By the assumptions on $\psi$ and regular variation,
  $\esp[|\phi_{n,0}|^q] = O(\bar{\HH}(u_n))$ for all $\phi\in\mcm_\psi$ and $q>0$ which implies
  $\esp[S_n(\phi)]=O(r_n\bar{H}(u_n))$.  Since $r_n=o(n)$, we have, for large enough~$n$,
  \begin{align*}
    \ind{|\bar{S}_{n,1}(\phi)|>\eta\sqrt{n\bar{\HH}(u_n)}}
    \leq  \ind{|{S}_{n,1}(\phi)|>\eta\sqrt{n\bar{\HH}(u_n)}/2}   \; .
  \end{align*}
  Since $\eta$ is arbitrary, we can remove the centering from the indicator. Furthermore,
  \begin{align}
    \frac{1}{r_n\bar{\HH}(u_n)}
    & \esp\left[\bar{S}_n^2(\phi)\ind{|{S}_n(\phi)|>\eta v_n}\right] \nonumber \\
    & = \frac{1}{r_n\bar{\HH}(u_n)} \esp\left[S_n^2(\phi)\ind{|{S}_n(\phi)|>\eta v_n}\right] +
      \frac{O(\{\esp[S_n(\phi)]\}^2)}{r_n\bar{\HH}(u_n)} \nonumber      \\
    & = \frac{1}{r_n\bar{\HH}(u_n)} \esp\left[S_n^2(\phi)\ind{|{S}_n(\phi)|>\eta v_n}\right] + O(r_n\bar{\HH}(u_n)) \; .
    \label{eq:proof-of-lindeberg-centering-1}
  \end{align}
  Hence, under condition~(\ref{eq:rate-condition-fidi}), it suffices to study the main term on the
  right hand side of (\ref{eq:proof-of-lindeberg-centering-1}), which is developed as $I_1+2I_2$
  with
  \begin{align*}
    I_1  = \frac{1}{r_n\bar{\HH}(u_n)} \sum_{j=1}^{r_n} \esp[\phi_{n,j}^2 \ind{|{S}_n(\phi)|>\eta v_n}] \; ,  \ \
    I_2  = \frac{1}{r_n \bar{\HH}(u_n)}\sum_{i=1}^{r_n}\sum_{j=i+1}^{r_n} \esp[\phi_{n,i}\phi_{n,j}
    \ind{|{S}_n(\phi)|>\eta v_n}]\; .
  \end{align*}
  By (\ref{eq:var-psi-sigma}) and the H\"{o}lder inequality, we have
  $\esp[|S_n(\phi)|] = O(\sqrt{r_n\bar{\HH}(u_n)})$.  Applying Markov's inequality and the
  boundedness of $\phi$, we obtain
  \begin{align*}
    I_1 & \leq \frac{1}{\eta v_n r_n\bar{\HH}(u_n)} \sum_{j=1}^{r_n} \esp[\phi_{n,j}^2|S_n(\phi)|]\\
        & = O\left(\frac{1}{ v_n}\right) \frac1{r_n\bar{\HH}(u_n)} \esp\left[ \left(\sum_{j=1}^{r_n} |\phi_{n,j}|\right)^2 \right]
          =  O\left(\frac{1}{ v_n}\right)  = o(1) \; .
  \end{align*}
  Fix a positive integer $L$. Since $\phi$ is bounded, we have
  \begin{align}
    I_2 & \leq \frac{1}{r_n \bar{\HH}(u_n)}\sum_{i=1}^{L} \sum_{j=i+1}^{r_n} \esp[|\phi_{n,i} \phi_{n,j}| \ind{|{S}_n(\phi)|>\eta v_n}]
          \label{eq:tep-Lindeberg-I2a} \\
        & \phantom{ \leq }
          + \frac{1}{r_n \bar{\HH}(u_n)} \sum_{i=L+1}^{r_n}\sum_{j=i+1}^{i+L} \esp[|\phi_{n,i} \phi_{n,j}| \ind{|{S}_n(\phi)|>\eta v_n}]
          \label{eq:tep-Lindeberg-I2ab}     \\
        & \phantom{ \leq }
          + \frac{1}{r_n \bar{\HH}(u_n)}\sum_{i=L+1}^{r_n} \sum_{j=i+L+1}^{r_n} \esp[|\phi_{n,i} \phi_{n,j}|] \; .
          \label{eq:tep-Lindeberg-I2c}
  \end{align}
  Since $\phi$ is bounded, the terms in \eqref{eq:tep-Lindeberg-I2a}
  and~(\ref{eq:tep-Lindeberg-I2ab}) are each bounded by
  \begin{align*}
    \frac{L\|\phi\|_\infty}{r_n \bar{\HH}(u_n)} \sum_{j=1}^{r_n} \esp[
  |  \phi_{n,j}| \ind{|{S}_n(\phi)|>\eta v_n}] = o(1) \; ,
  \end{align*}
  by the same argument as for $I_1$. Thus,
  \begin{align}
    \label{eq:tep-Lindeberg-I2b}
    I_2 = o(1)  + \frac{2}{\bar{\HH}(u_n)}\sum_{i=L+1}^{r_n} \esp[|\phi_{n,0} \phi_{n,i}|] \; .
  \end{align}
  By Condition \ref{eq:anticlustering-weighted}, the last expression in (\ref{eq:tep-Lindeberg-I2b})
  can be made arbitrarily small by choosing $L$ large enough.
\end{proof}

We extend \Cref{lem:asympt-neglig-bounded} to the case of unbounded functions, at the cost of the
extra restriction (\ref{eq:rate-condition-fidi-unbounded}) on the sequence $\{r_n\}$.
\begin{lemma}
  \label{lem:asympt-neglig-unbounded}
  Let $\psi$ be a non negative measurable function which vanishes in a neighborhood of zero, such that Conditions
  \ref{eq:anticlustering-weighted}, (\ref{eq:psi-second-moment-1}), (\ref{eq:rate-condition-fidi})
  and (\ref{eq:rate-condition-fidi-unbounded}) hold for the same $\delta\in (0,1]$.  Then for all
  $\eta>0$ and all $\phi\in\mcm_\psi$,
  \begin{align*}
  &  \lim_{n\to0} \frac1{r_n\bar{\HH}(u_n)}  \esp \left[ {S}_{n,1}^2(\phi)
      \ind{|{S}_{n,1}(\phi)|>\eta\sqrt{n\bar{\HH}(u_n)}} \right] \\
&=    \lim_{n\to0} \frac1{r_n\bar{\HH}(u_n)} \esp \left[ \bar{S}_{n,1}^2(\phi)
      \ind{|\bar{S}_{n,1}(\phi)|>\eta\sqrt{n\bar{\HH}(u_n)}} \right] = 0 \; .
  \end{align*}
\end{lemma}

\begin{proof}
  We follow closely the proof of~\Cref{lem:asympt-neglig-bounded} with appropriate modifications.
  Recall that we have set $v_n = \sqrt{n\bar{H}(u_n)}$. Since $\phi\in\mcm_\psi$,
  \Cref{lem:covar-existence} implies that $\esp[S_n(\phi)] = O(\sqrt{r_n\tail{H}(u_n)})$ and thus
  the centering can be removed inside the indicator. The calculations leading
  to~(\ref{eq:proof-of-lindeberg-centering-1}) are still valid in the unbounded case.  Since
  $\delta\in(0,1]$, we have by Markov inequality,
  \begin{align*}
    I_1 & = O\left( \frac{1}{v_n^{\delta}}\right)  \frac1{r_n\bar{\HH}(u_n)} \sum_{j=1}^{r_n} \esp[\phi_{n,j}^2|S_n(\phi)|^{\delta}]\\
        &  = O\left( \frac{1}{v_n^{\delta}}\right)  \frac{1}{r_n\bar{\HH}(u_n)} \sum_{i=1}^{r_n} \sum_{j=1}^{r_n}
          \esp[\phi_{n,i}^2|\phi_{n,j}|^\delta] =    O\left(\frac{r_n}{v_n^\delta}\right) = o(1) \; ,
  \end{align*}
  by~(\ref{eq:psi-second-moment-1}) and~(\ref{eq:rate-condition-fidi-unbounded}). As for $I_2$, the
  second term in~(\ref{eq:tep-Lindeberg-I2b}) is handled again by Condition
  \ref{eq:anticlustering-weighted}.  Applying Markov inequality, the term
  in~(\ref{eq:tep-Lindeberg-I2ab}) is bounded by
 \begin{align*}
   \frac{1}{\eta v_n^{\delta}r_n \bar{\HH}(u_n)} \sum_{i=\ell+1}^{r_n}
   \sum_{j=i+1}^{i+\ell}\sum_{k=1}^{r_n} \esp[|\phi_{n,i}| |\phi_{n,j}| |\phi_{n,k}|^{\delta}] \leq
   \frac{\ell r_n}{\eta v_n^{\delta}}\frac{\esp[|\phi_{n,0}|^{2+\delta}]}{ \bar{\HH}(u_n)}\;,
 \end{align*}
 on account of the extended H\"{o}lder inequality with $p=q=2+\delta$ and $r=(2+\delta)/\delta$.
 This again is $o(1)$ by~(\ref{eq:psi-second-moment-1})
 and~(\ref{eq:rate-condition-fidi-unbounded}). The term  in~(\ref{eq:tep-Lindeberg-I2a}) is
 treated analogously.
\end{proof}

\section{Proof of   the results for functions of Markov chains}
\label{sec:proof-markov}
Let $\{X_j,j\in\Zset\}$ be as in \Cref{sec:results}. We will apply the results of \Cref{sec:general}
to the sequence $\{\xi_j,j\in\Zset\}$ defined as $\xi_j=\vectorbold{X}_{j,j+h}=(X_j,\ldots,X_{j+h})$
which is also regularly varying. Since the distribution of $X_0$ satisfies the
balanced tail condition and the right tail of $X_0$ is not trivial, $X_0$ and $|\xi_0|$ are tail
equivalent.

We first recall some consequences of the geometric drift condition. Under
condition~(\ref{eq:small-set}), the chain $\{\mby_j\}$ can be embedded into an extended Markov chain
$\{(\mby_j,B_j)\}$ such that the latter chain possesses an atom~$A$, that is
$\bar{P}(s,\cdot)=\bar{P}(t,\cdot)$ for every $s,t\in A$, where $\bar{P}$ is the transition kernel
of the extended chain. This existence is due to the Nummelin splitting technique (see
\cite[Chapter~5]{meyn:tweedie:2009}). Denote by $\esp_A$ the expectation conditionally to
$(\mby_0,B_0)\in A$ and let $\return{A}$ be the first return time to $A$ of the chain
$\{(\mby_j,B_j),j\geq0\}$.  Note that $\return{A}$ is a stopping time with respect to the extended
chain, but not with respect to the chain $\{\mby_j\}$.  We assume that the extended chain is defined
on the original probability space $(\Omega,\mcf,\pr)$ and that the extended chain is stationary
under $\pr$. Then, by \cite[Theorem~15.4.1]{meyn:tweedie:2009} for $q_0$ as in
\Cref{hypo:drift-small}, there exist $\kappa>1$ and a constant $\constant$ such that for all
$y\in E$, % \phnote{MT-th.15.2.5 for $C$}
\begin{align}
  \label{eq:moment-geometrique}
  \esp \left[\sum_{j=1}^{\tau_A} \kappa^{j} |X_j|^{q_0} \mid \mby_0=y \right]\leq c\esp
  \left[\sum_{j=1}^{\tau_A} \kappa^{j} V(\mby_j) \mid \mby_0=y \right] \leq \constant V(y) \; .
\end{align}
By Jensen's inequality, this implies that for all $q_1\leq q_0$, there exists $\kappa_1\in(1,\kappa)$
such that
\begin{align}
  \label{eq:moment-geometrique-jensen}
  \esp \left[\sum_{j=1}^{\tau_A} \kappa_1^{j} |X_j|^{q_1} \mid \mby_0=y \right] \leq \constant  V^{q_1/q_0}(y)  \; .
\end{align}
Moreover, Kac's formula \cite[Theorem 10.0.1]{meyn:tweedie:2009} gives an expression of the
stationary distribution in terms of the return time to $A$. For every bounded measurable function
$f$, it holds that
\begin{align}
  \label{eq:kac}
  \esp[f(\mby_0)] = \frac1{\esp_{A}[\tau_A]} \esp_{ A} \left[ \sum_{j=0}^{\tau_A-1} f(\mby_j)
  \right] \; .
\end{align}
Since $V\geq1$, the inequality \eqref{eq:moment-geometrique} integrated with respect to the
stationary distribution implies that $\esp[\kappa^{\tau_A}]<\infty$.

\subsection{Checking the \anticlustering\ condition}  %%%% olivier's version
\label{sec:proof-of-anticluster}

In the present context, the anticlustering condition~\ref{eq:anticlustering-weighted} can be re-written as
\begin{align}
 \label{eq:anticlustering-markov}
 \lim_{L\to\infty} \limsup_{n\to\infty} \frac{1}{\tail{F}(u_n)} \sum_{j=L+1}^{r_n} \esp \left[
   |\psi({s\vectorbold{X}_{0,h}}/{u_n})| | \psi (t\vectorbold{X}_{j,j+h}/{u_n}) | \right] = 0 \; ,
\end{align}
where $F$ is the distribution function of $X_0$.

Let $|\cdot|$ denote an arbitrary norm on $\Rset^{h+1}$.  In this section, we prove that for all $\epsilon>0$
\Cref{hypo:drift-small} implies the condition~(\ref{eq:anticlustering-markov}) for  the function $\psi_\epsilon$
defined by
\begin{align*}
  \psi_\epsilon(\bx) = {|\bx|^{q_0/2}} \ind{|\bx|>\epsilon}\;.
\end{align*}
This will be done in~\Cref{prop:conditiondhs-via-drift} below. First we introduce some notation and
prove two preliminary results. For $0<s<  \infty$ define
\begin{align*}
  Q_n(s) & = \frac{1}{u_n^{q_0} \tail{F}(u_n)} \esp\left[V(\mby_0) \ind{|X_0|> u_n s}\right]    \; .
\end{align*}

\begin{lemma}
  \label{lem:sum-cov-psi}
%\philippe{2018/07/02: new version without $h$}
Let \Cref{hypo:drift-small} holds. For every $s_0>0$, there exists a constant $C_0>0$ and
  $\kappa_0>1$ such that for $q_1+q_2\leq q_0$ and $s\geq s_0$,
  \begin{align}
    \frac{1}{\tail{F}(u_n)} \esp \left[\sum_{j=L}^{\tau_A} \ind{|X_0|>su_n}
      |X_{0}/u_n|^{q_1} |X_j/u_n|^{q_2} \right] \leq C_0 \kappa_0^{-L} {Q_n^{(q_1+q_2)/q_0}(s)} \; . \label{eq:sum-cov-psi}
  \end{align}
\end{lemma}

\begin{proof}
  Let $\kappa$ be as in (\ref{eq:moment-geometrique}).  Let the left hand side of
  (\ref{eq:sum-cov-psi}) be denoted by $S_n(s)$.
  Conditioning on $\mby_0$ and applying (\ref{eq:moment-geometrique-jensen}), we obtain that there exists
  $\kappa_0\in (1,\kappa)$ such that
  \begin{align*}
    S_{n}(s)
    & \leq \constant \cdot \kappa_0^{-L} \frac{1}{u_n^{q_2}\tail{F}(u_n)} \esp \left[ V^{q_2/q_0}(\mby_0) |X_0/u_n|^{q_1} \ind{|X_0|>su_n} \right]  \\
    & \leq \constant \cdot \kappa_0^{-L} \frac{\tail{F}(u_ns)}
      {u_n^{q_2}\tail{F}(u_n)} \esp \left[ V^{q_2/q_0}(\mby_0) |X_0/u_n|^{q_1} \mid |X_0|>su_n \right]   \\
    & \leq \constant \cdot \kappa_0^{-L} \frac{\tail{F}(u_ns)}
      {u_n^{q_1+q_2}\tail{F}(u_n)} \esp \left[ V^{(q_1+q_2)/q_0}(\mby_0)  \mid |X_0|>su_n \right]\;.
  \end{align*}
  Applying Jensen's inequality  to the conditional distribution given $|X_0|>u_ns$, we
  obtain
  \begin{align}
    \label{eq:sn1}
    S_{n}(s)
    & \leq \constant \cdot \kappa_0^{-L} \frac{\tail{F}(u_ns)}
      {u_n^{q_1+q_2} \tail{F}(u_n)} \left(\esp \left[ V(\mby_0)  \mid |X_0|>su_n \right]  \right)^{(q_1+q_2)/q_0} \nonumber \\
    & =  \constant \cdot \kappa_0^{-L} \frac{\tail{F}(u_ns)}
      {u_n^{q_1+q_2} \tail{F}(u_n)} \left(\esp \left[ V(\mby_0)  \ind{ |X_0|>su_n} \right]/\tail{F}(u_n)  \right)^{(q_1+q_2)/q_0}\left(\frac{\tail{F}(u_n)}{\tail{F}(u_n s)}\right)^{(q_1+q_2)/q_0}\nonumber\\
    & = \constant\cdot\kappa_0^{-L}  \left(\frac{\tail{F}(u_ns)} {\tail{F}(u_n)} \right)^{1-(q_1+q_2)/q_0}  Q_n^{(q_1+q_2)/q_0}(s) \;.
  \end{align}
  This yields~(\ref{eq:sum-cov-psi}) since $\tail{F}(u_ns)/\tail{F}(u_n)$ is uniformly bounded on
  $[s_0,\infty)$ and $1-(q_1+q_2)/q_0>0$.
\end{proof}

\begin{lemma}
  \label{lem:reste}
  % \philippe{2018/07/02: New version without $h$ in the indicator but in the upper bound of the
  % sum}
  If \Cref{hypo:drift-small} holds, $r_n\tail{F}(u_n)=o(1)$, and $q_1+q_2\leq q_0$, then
  \begin{align}
    \label{eq:reste}
    \frac{1}{\tail{F}(u_n)} \esp\left[  \sum_{j=\tau_A+1}^{r_n{+h}} \ind{|X_0|>su_n}
    \ind{|X_j|>su_n} |X_0/u_n|^{q_1} |X_j/u_n|^{q_2} \right]  = o(1) \; ,
  \end{align}
  uniformly with respect to $s\geq s_0$.
\end{lemma}

\begin{proof}
  Let the left hand side of~(\ref{eq:reste}) be denoted by $R_n(s)$. Then, by the strong Markov
  property,
  \begin{align*}
     R_n(s) &  \leq \frac{1}{\tail{F}(u_n)} \esp\left[      \ind{|X_0|>u_ns} |X_0/u_n|^{q_1}
      \esp\left[\sum_{j=\tau_A+1}^{r_n{+h}} \ind{|X_j|>u_ns_0 } |X_j/u_n|^{q_2} \mid \mathbb{Y}_{\tau_A} \right] \right]  \\
    & \leq \frac{1}{\tail{F}(u_n)} \esp\left[ \ind{|X_0|>u_ns_0} |X_0/u_n|^{q_1}
       \esp \left[\sum_{j=\tau_A+1}^{r_n+\tau_A{+h}} \ind{|X_j|>u_ns_0 } |X_j/u_n|^{q_2} \mid \mathbb{Y}_{\tau_A} \right] \right] \\
    & \leq \frac{1}{\tail{F}(u_n)} \esp\left[ \ind{|X_0|>u_ns_0} |X_0/u_n|^{q_1} \right]
      \esp_A \left[\sum_{j=1}^{r_n{+h}} \ind{|X_j|>u_ns_0 } |X_j/u_n|^{q_2} \right]  \; .
  \end{align*}
  By classical regenerative arguments, Kac's formula~(\ref{eq:kac}) and regular variation, we obtain
  as $n\to\infty$ {(and since $h$ is fixed)},
  \begin{align*}
    \esp_{A} \left[ \sum_{j=1}^{r_n{+h}} \ind{|X_j|>u_n s_0}  |X_j/u_n|^{q_2} \right]
    & \sim    \frac{r_n}{\esp_{A}[\tau_A]} \esp_A\left[\sum_{j=1}^{\tau_A}\ind{|X_j|>u_n s_0 } |X_j/u_n|^{q_2}\right] \\
    & \leq r_n\tail{F}(u_n) \frac{\esp[\ind{|X_0|>u_n{s_0}}|X_0/u_n|^{q_2}]}{\tail{F}(u_n)} =
    O(r_n\tail{F}(u_n)) \; .
  \end{align*}
  This yields, applying again the drift condition, Condition~(\ref{eq:def-Q}) and Jensen's
  inequality,
  \begin{align*}
   \sup_{s\geq s_0} R_{n}(s) & = O(r_n) \esp\left[\ind{|X_0|>u_ns_0} |X_0/u_n|^{q_1} \right] \\
            & = O(r_n\tail{F}(u_n)) \frac{\esp\left[V^{q_1/q_0}(\mby_0) \mid |X_0|>u_ns_0\right]} {u_n^{q_1}} \\
          &= O(r_n\tail{F}(u_n)) \left(\frac{\esp\left[V(\mby_0) \mid |X_0|>u_ns_0\right]} {u_n^{q_0}}\right)^{q_1/q_0} = o(1) \; .
  \end{align*}
\end{proof}

\begin{lemma}
  \label{prop:conditiondhs-via-drift}
  Let \Cref{hypo:drift-small} hold.  Then $\mathcal{S}(u_n,r_n,\psi_{\epsilon})$ holds.
\end{lemma}

\begin{proof}
  For $j\geq h$, we have
  \begin{align}
    0 & \leq {|\psi_\epsilon(\vectorbold{X}_{0,h}/u_n)| |\psi_\epsilon(\vectorbold{X}_{j,j+h}/u_n)|} \nonumber \\
      & \leq \constant u_n^{{-q_0}} \sum_{i_1,i_2,i_3,i_4=0}^h \ind{|X_{i_1}|> u_n \epsilon }
        |X_{i_2}|^{q_0/2} \ind{|X_{j+i_3}|>u_n \epsilon} |X_{j+i_4}|^{q_0/2}\; .  \label{eq:first-expansion}
  \end{align}
  For all $i,i'$, we can write
  \begin{align*}
    & \ind{|X_{i}|> u_n \epsilon} |X_{i'}/u_n|^{q_0/2}  \\
    & \phantom{\ind{}} =   \ind{|X_{i}|> u_n \epsilon}  \ind{|X_{i'}| \leq  u_n \epsilon} |X_{i'}/u_n|^{q_0/2}
      +  \ind{|X_{i}|> u_n \epsilon}  \ind{|X_{i'}|> u_n \epsilon} |X_{i'}/u_n|^{q_0/2} \\
    & \phantom{\ind{}} \leq  \ind{|X_{i}|> u_n \epsilon}  \epsilon^{q_0/2}
      + \ind{|X_{i'}|> u_n \epsilon} |X_{i'}/u_n|^{q_0/2} \\
    & \phantom{\ind{}} \leq \ind{|X_{i}|>u_n \epsilon} |X_{i}/u_n|^{q_0/2} + \ind{|X_{i'}|>u_n \epsilon} |X_{i'}/u_n|^{q_0/2}  \; .
  \end{align*}
  Thus, we can restrict the sum in (\ref{eq:first-expansion}) to the set of indices
  $(i_1,i_2,i_3,i_4)$ such that $i_1=i_2$ and $i_3=i_4$.
  For $L>h$ and $i,i'\leq h$, we have by \Cref{lem:sum-cov-psi,lem:reste} and by stationarity,
  \begin{align*}
    \frac{1}{\tail{F}(u_n)}
    &  \esp \Big [\sum_{j={2L}}^{r_n} \ind{|X_{i}|>u_n s } \ind{|X_{j+i'}|>u_n s } |X_{i}/u_n|^{q_0/2}  |X_{j+i'}/u_n|^{q_0/2} \Big] \nonumber    \\
    & \leq \frac{1}{\tail{F}(u_n)} \esp \Big [\sum_{j={2L}}^{r_n}
      \ind{|X_0|>u_n s } \ind{|X_{j+i'-i}|>u_n s } |X_0/u_n|^{q_0/2}  |X_{j+i'-i}/u_n|^{q_0/2} \Big]  \nonumber    \\
    & \leq \frac{1}{\tail{F}(u_n)} \esp \Big [\sum_{j={2L-h}}^{r_n{+h}}
      \ind{|X_0|>u_n s } \ind{|X_{j}|>u_n s } |X_0/u_n|^{q_0/2}  |X_{j}/u_n|^{q_0/2} \Big]  \nonumber    \\
    & \leq \frac{1}{\tail{F}(u_n)} \esp\left[\sum_{j=L}^{\tau_A} \ind{|X_0|>u_n s}
      \ind{|X_{j}|>u_n s} |X_0/u_n|^{q_0/2} |X_{j}/u_n|^{q_0/2} \right]     \\
    & + \frac{1}{\tail{F}(u_n)} \esp\left[ \sum_{j=\tau_A+1}^{r_n{+h}} \ind{|X_0|>u_n s}
      \ind{|X_{j}|>u_n s} |X_0/u_n|^{q_0/2} |X_{j}/u_n|^{q_0/2} \right]     \\
    & \leq \constant \;\kappa_0^{-L} Q_n(s) + o(1) \; .
  \end{align*}
  where the $o(1)$ term is uniform \wrt\ $s\geq s_0$.  The bound (\ref{eq:def-Q}) in
  \Cref{hypo:drift-small} implies that $Q_n$ is asymptotically uniformly bounded on
  $[s_0,\infty)$. Therefore,
  \begin{align*}
    \limsup_{n\to\infty} \frac{1}{\tail{F}(u_n)} \esp \Big [\sum_{j=L}^{r_n}  \ind{|X_0|> u_n s }
    \ind{|X_j|> u_n s} |X_0/u_n|^{q_0/2}  |X_j/u_n|^{q_0/2} \Big]  = O(\kappa_0^{-L}) \; .
  \end{align*}
  Since $\kappa_0>1$, this proves (\ref{eq:anticlustering-markov}).
\end{proof}

\subsection{Proof of \Cref{theo:fidi}}
\label{sec:proof-tep-markov}

Fix $q\leq q_0/{(2+\delta)}$.  We apply \Cref{theo:tail-array-fidi} to the sequence
$\xi_j = (X_j,\dots,X_{j+h})$, $j\in\Zset$ in order to prove that $\TEPmult_{n} \fidi \TEPmult$ on
$\mcm_{h_{q,\epsilon}}$ for each $\epsilon>0$, where $h_{q,\epsilon}$ is defined  on $\Rset^{h+1}$ by
\begin{align}
  \label{eq:psi-epsilon-definition-a}
  h_{q,\epsilon}(\bx) = {|\bx|^{q}} \ind{|\bx|>\epsilon}\;.
\end{align}
Since $\mcl_q = \cup_{\epsilon>0} \mcm_{h_{q,\epsilon}}$, this will prove \Cref{theo:fidi}.
\begin{enumerate}[(i),wide=0pt]
\item \label{item:i-of-proof-of-fidi}  \Cref{hypo:drift-small} implies $\mathcal{S}(u_n,r_n,h_{q,\epsilon})$
  (cf.~\Cref{prop:conditiondhs-via-drift} and note that $h_{q_0/2,\epsilon}=\psi_{\epsilon}$) and $\beta$-mixing with geometric rate.  Therefore we can
  choose $r_n=\log^{1+\eta}(n)$ and $\ell_n= c \log(n)$ for $c$ large enough so that
  (\ref{eq:beta-fidi}) holds.
\item Since $q(2+\delta) \leq q_0$, assumptions~(\ref{eq:g-bound-v}) and~(\ref{eq:def-Q}) imply that
  Condition~(\ref{eq:psi-second-moment-1}) holds for $\psi=h_{q,\epsilon}$.
\item With $r_n$ as above, Conditions~(\ref{eq:condition-un-eta}) and (\ref{eq:missing-condition})
  are exactly (\ref{eq:rate-condition-fidi}) and (\ref{eq:rate-condition-fidi-unbounded}).
\end{enumerate}

\subsection{Proof of \Cref{theo:our-entropy}}
\label{sec:prooftheofuncmarkov}
Let $\ell_0(\Rset^{h+1})$ be the set of $\Rset^{h+1}$-valued sequences $\bx = (\bx_j)_{j\in\Zset}$
such that $\lim_{|j|\to\infty} |\bx_j|=0$.  Let~$\mch_q$ be the set of functions $f$ defined on
$\ell_0(\Rset^{h+1})$ for which there exists $\phi\in\mcl_q$ such that
\begin{align*}
  f(\bx) = \sum_{j\in\Zset} \phi(\bx_j) \; ,  \ \ \bx \in\ell_0(\Rset^{h+1}) \; .
\end{align*}
Since functions in $\mcl_q$ vanish in a neighborhood of zero, the series has finitely many non zero
terms and the function $\phi$ is uniquely determined by $f$ and will be denoted $\phi^f$.  We define
a pseudometric $\rho$ on $\mch_q$ (with $q<q_0/(2+\delta)$ by
\begin{align}
  \label{eq:def-semimetric-rho}
  \rho^2(f,g) & = \rho_h^2(\phi^f,\phi^g) = \numultseq{0,h}(\{\phi^{f}-\phi^{g}\}^2)  \; .
\end{align}
Let~$\{r_n\}$ be as in \ref{item:i-of-proof-of-fidi} of the proof of \Cref{theo:fidi} and let
$m_n=[n/r_n]$. Set
$\blockX_{n,i} = u_n^{-1}(\bX_{(i-1)r_n+1,(i-1)r_n+h+1},\dots,\bX_{ir_n,ir_n+h})$ and identify
  it with an element of $\ell_0(\Rset^{h+1})$ by adding zeros on both sides. Then
\begin{align*}
  \mbz_{n,i}(f)
  & = f(\blockX_{n,i}) = \sum_{j=(i-1)r_n+1}^{ir_n} \phi^f(u_n^{-1} \vectorbold{X}_{j,j+h})  \; , \ \
    \mbz_n(f) = \sum_{i=1}^{m_n} \mbz_{n,i}(f) \; , \\
  \bar\mbz_n(f)
  & =  \frac{1}{\sqrt{n\tail{F}(u_n)}}(\mbz_n(f)-\esp[\mbz_n(f)]) \; .
\end{align*}
{Let $\mbz$ be the Gaussian process on $\mch_q$ defined by $\mbz(f)=\TEPmult(\phi^f)$.
Under the assumptions of \Cref{theo:fidi} and with $r_n = \log^{1+\eta}(n)$, we have
\begin{align*}
  \frac1{\sqrt{n\bar{F}(u_n)}}   \sum_{i=r_n[n/r_n]+1}^n \left\{\phi^f(u_n^{-1}
  \vectorbold{X}_{j,j+h}) - \esp[\phi^f(u_n^{-1} \vectorbold{X}_{0,h})]\right\} = o_P(1) \;.
\end{align*}
Moreover, since the envelope function belongs to $\mcl_q$, we can apply \Cref{lem:covar-existence}
(in view of \Cref{prop:conditiondhs-via-drift}) and we obtain
\begin{align*}
  \sup_{f\in\classF{G}} \left| \TEPmult_n(\phi^f)-\bar\mbz_n(f) \right|
  & \leq \frac1{\sqrt{n\tail{F}(u_n)}} \sum_{j=m_nr_n+1}^n \Phi_\mcg(X_{n,j}/u_n)
    + \frac{r_n\bar{F}(u_n)}{\sqrt{n\tail{F}(u_n)}} \frac{\esp[\Phi_\mcg(X_{n,0}/u_n) ]}{\bar{F}(u_n)} \\
  & = O_P\left(\sqrt{\frac{r_n}n}\right) + O\left(\frac{r_n\bar{F}(u_n)}{\sqrt{n\tail{F}(u_n)}}\right) = o_P(1) \; .
\end{align*}
Thus $\bar\mbz_n(f) = \TEPmult_{n}(\phi^f) + o_P(1)$ uniformly on $\mch_q$ and $\bar\mbz_n\fidi\mbz$ on $\mch_q$.}

We define the subclass $\classF{G}$ of $\mch_q$ associated to the subclass
$\mcg$ of $\mcl_q$ by  $\classF{G}=\{f:f(\bx) = \sum_{j\in\Zset} \phi(\bx_j)$,
$\phi\in\mcg\}$. {
Arguing as in \cite[Proof of Theorem~2.8]{drees:rootzen:2010}, in order to prove weak convergence of $\bar\mbz_n$, it suffices to prove the tightness of
the process $\mbz_n^*$ summing independent copies of $\mbz_{n,i}^*$ of $\mbz_{n,i}$ indexed by the class $\classF{G}$.  For this purpose, we apply
\Cref{theo:VW2.11.1}.}

\begin{enumerate}[$-$,wide=0pt]
\item The pointwise separability of $\mcg$ implies that  $\classF{G}$ is also pointwise separable.
\item The property~(\ref{eq:def-semimetric-rho}) yields that $(\classF{G},\rho)$ is totally bounded
  since $(\mcg,\rho_h)$ is totally bounded by assumption.

\item Since the envelope function $\Phi_\mcg$ is assumed to be in $\mcl_q$ and
  $\mcl_q=\cup_{\epsilon>0} \mcm_{h_q,\epsilon}$, \Cref{lem:asympt-neglig-unbounded} implies that
  the Lindeberg condition (\ref{eq:lindeberg-envelope}) holds.

\item We now  check (\ref{eq:continuite-L2}).  For $f,g\in\classF{G}$, we have
  \begin{align}
    \label{eq:conv-unif-metric-rho}
    \rho(f,g) & = \lim_{n\to\infty} \frac1{\tail{F}(u_n)} \esp[\{\phi^f(\bX_{0,h}/u_n)-\phi^g(\bX_{0,h}/u_n)\}^2] \; ,
  \end{align}
  Since the envelope function of $\mcg$ is in $\mcl_q$, there exists $\epsilon>0$ (which
    depends only on $\mcg$) such that $|\phi^f|\vee|\phi^g|\leq \epsilon^{-1} h_{q,\epsilon}$
    (defined in~(\ref{eq:psi-epsilon-definition-a})).  Set
    $\phi_{n,j} = \phi^f(\bX_{j,j+h}/u_n)-\phi^g(\bX_{j,j+h}/u_n)$.  For every integer $L>0$, by
    stationarity, we have
  \begin{align*}
    \frac{1}{r_n \tail{F}(u_n)} & \esp [ \{\mbz_{n,1}(f) - \mbz_{n,1}(g)\}^2] \\
    & \leq \frac{2}{ \tail{F}(u_n)}\sum_{j=0}^L \esp \left[ |\phi_{n,0}| |\phi_{n,j}|
    \right] + \frac{2}{ \tail{F}(u_n)} \sum_{j=L+1}^{r_n}
    \esp \left[ |\phi_{n,0}| |\phi_{n,j}| \right] \\
    & \leq \frac{2(L+1)\esp[\phi_{n,0}^2]} {\tail{F}(u_n)} + \frac{\constant}{ \tail{F}(u_n)}
    \sum_{j=L+1}^{r_n} \esp \left[ h_{q,\epsilon}(\vectorbold{X}_{0,h}/u_n) h_{q,\epsilon}(\vectorbold{X}_{j,j+h}/u_n) \right] \; .
  \end{align*}
  By \Cref{prop:conditiondhs-via-drift}, for every $\eta>0$, we can choose $L$ such that
  \begin{align}
    \limsup_{n\to\infty} \frac{1}{ \tail{F}(u_n)} \sum_{j=L+1}^{r_n} \esp \left[
      h_{q,\epsilon}(\vectorbold{X}_{0,h}/u_n) h_{q,\epsilon}(\vectorbold{X}_{j,j+h}/u_n) \right] \leq \eta \; . \label{eq:limsup1}
  \end{align}
  Applying this bound and the assumption (\ref{eq:continuite-pour-notre-theorem}) yields, for any
  sequence $\delta_n$ decreasing to zero,
  \begin{align}
    \limsup_{n\to\infty} \sup_{f,g\in\classF{G}\atop \rho(f,g)\leq \delta_n}
    \frac{1}{r_n \tail{F}(u_n)}\esp\left[\left(\mbz_{n,1}(f) - \mbz_{n,1}(g)\right)^2\right] \leq \eta \; .
  \end{align}
  Since $\eta$ is arbitrary, this proves (\ref{eq:continuite-L2}).
\item
Since $\mcg$ is linearly ordered, so is $\classF{G}$ thus it is a VC subgraph class and the entropy
condition (\ref{eq:randomentropy}) holds by
\cite[Theorem~3.7.37]{gine:nickl:2016}.
\end{enumerate}
We have checked all the assumptions of \Cref{theo:VW2.11.1}. { Therefore, $\mbz_n^*$ and hence
$\bar\mbz_n$ converges to $\mbz$ in $\ellinfty(\classF{G})$ and since $\TEPmult(\phi^f)=\mbz(f)$ by definition, this proves our result.}
\subsection{Proof of \Cref{coro:tepusual,coro:randomthreshold}}
\label{sec:prooftheofeasible}

We first apply \Cref{theo:our-entropy} with
$\mcg_I=\{I_s = \1{(s,\infty)\times\Rset^h}, s\in[s_0,t_0]\}$ for $0<s_0\leq 1 \leq t_0$. Then
\begin{align*}
  \rho_h(I_s,I_t) = |s-t|^{-\alpha} \leq \alpha s_0^{-\alpha-1} |s-t| \; .
\end{align*}
The class $\mcg_I$ is pointwise separable, linearly ordered and totally bounded for the metric
$\rho_h$. Condition (\ref{eq:continuite-pour-notre-theorem}) holds by regular variation and the
uniform convergence theorem. The envelope of $\mcg_I$ is $I_{s_0}$ which belongs to $\mcl_q$.  This
proves \Cref{coro:tepusual}.

We now prove \Cref{coro:randomthreshold}.  Set
\begin{align*}
  B_{n,1}(s) = \sqrt{k}\{\esp[M_n(I_s)]-s^{-\alpha}\} \; , \ \  B_{n,2}(s) = \sqrt{k}\{\esp[M_n(\phi_s)]-\numultseq{0,h}(\phi_s)\} \; .
\end{align*}
By conditions~(\ref{eq:def-tep-bias-uniform}) and~(\ref{eq:bias-psi}), we have
$\lim_{n\to\infty}\sup_{s\in[s_0,t_0]} B_{n,i}(s)=0$, $i=1,2$.  Using this bound and
$\numultseq{0,h}(\phi_s) = s^{-\alpha}\numultseq{0,h}(\phi)$, we obtain after some algebra,
\begin{align*}
  \sqrt{k}\left(  \frac{M_n(\phi_{s})}{M_n(I_{s})} - \numultseq{0,h}(\phi) \right)
  & = \frac{\TEPmult_n(\phi_s) - \numultseq{0,h}(\phi) \TEPmult_n(I_s)-\numultseq{0,h}(\phi)B_{n,1}(s)+ B_{n,2}(s)}{M_n(I_s)} \\
  & = \frac{\TEPmult_n(\phi_s) - \numultseq{0,h}(\phi) \TEPmult_n(I_s)+o(1)}{M_n(I_s)} \; ,
\end{align*}
the term $o(1)$ begin uniform in $s \in[s_0,t_0]$ and $\phi\in\mcg_0$. Moreover,
\begin{align*}
  M_n(I_s)
  & = s^{-\alpha} + k^{-1/2} \TEPmult_n(I_s) +  k^{-1/2}B_{n,1}(s) = s^{-\alpha}  + o_p(1) \; ,
\end{align*}
again uniformly in $s \in[s_0,t_0]$. Therefore,
\begin{align}
  \sqrt{k}\left(  \frac{M_n(\phi_{s})}{M_n(I_{s})} - \numultseq{0,h}(\phi) \right)
  & = \frac{\TEPmult_n(\phi_s) - \numultseq{0,h}(\phi) \TEPmult_n(I_s)+o(1)}{s^{-\alpha}+o_P(1)} \; ,
\end{align}
the terms $o_P(1)$ begin uniform in $s \in[s_0,t_0]$ and $\phi\in\mcg_0$.
Set  $\zeta_n = X_{n:n-k}/u_n$. Since $ M_n(I_{\zeta_n})=1$,
\begin{align*}
  \widehat{\TEPmult}_n(\phi) = \sqrt{k} \left(\frac{M_n(\phi_{\zeta_n})}{M_n(I_{\zeta_n})} -\numultseq{0,h}(\phi)\right)\; ,
\end{align*}
and $\zeta_n\convprob1$ (see comments after \Cref{coro:tepusual}), we finally obtain that
\begin{align*}
\widehat{\TEPmult}_n(\phi)
  =   \frac{\TEPmult_n(\phi_{\zeta_n}) - \numultseq{0,h}(\phi) \TEPmult_n(I_{\zeta_n})+o(1)}{\zeta_n^{-\alpha}+o_P(1)}
  \convweak \TEPmult(\phi) - \numultseq{0,h}(\phi) \TEPmult(\1{[1,\infty)\times\Rset^h}) \; ,
\end{align*}
on $\ellinfty(\mcg_0)$.
\subsubsection{Proof for \Cref{sec:cluster}}
\label{sec:proof-clusterindex}

By homogeneity of $\numultseq{0,h}$, the semimetric $\rho_h$ on the class $\mcg$ for
this example becomes, for $s_0<s<t$,
\begin{align*}
  \rho_h^2(\phi_s,\phi_t)  & = \numultseq{0,h}(\{\ind{x_0+\cdots+x_h>s}-\ind{x_0+\cdots+x_h>t}\}^2) \\
                           & = \numultseq{0,h}(A_h) (s^{-\alpha}-t^{-\alpha}) \leq \constant \cdot (t-s) \; .
\end{align*}
Thus $(\mcg,\rho_h)$ is totally bounded. Moreover, by regular variation and the uniform
convergence theorem, the convergence
\begin{align*}
  \lim_{n\to\infty}  \frac1{\bar{F}(u_n)} \esp\left[ \{\ind{X_0+\cdots+X_h>su_n} -\ind{X_0+\cdots+X_h>tu_n} \}^2\right]
  & = \rho_h^2(\phi_s,\phi_t)
\end{align*}
is uniform on $[s_0,t_0]^2$. Thus (\ref{eq:continuite-pour-notre-theorem}) holds.

\subsubsection{Proof for \Cref{sec:expected-shortfall}}
\label{sec:proof-shorfall}
Since $\alpha>2$, the semimetric $\rho_h$ on the class $\mcg$ for to this example becomes, for
$s_0<s<t$,
\begin{align*}
  \rho_h^2(\phi_s,\phi_t)
  & = \esp[\{s^{-1}Y_h\ind{Y_0>s}-t^{-1}Y_h\ind{Y_0>t}\}^2] \\
& = \frac{\alpha\esp[\Theta_h^2]}{\alpha-2} \left\{(s^{-\alpha}+t^{-\alpha}-2\frac{1}{st}(s\vee t)^{-\alpha+2}\right\} \\
  & \leq \frac{\alpha\esp[\Theta_h^2]}{\alpha-2} (s^{-\alpha}-t^{-\alpha}) \leq \constant \cdot (t-s) \; .
\end{align*}
Thus $(\mcg,\rho_h)$ is totally bounded. Moreover, by regular variation and the uniform
convergence theorem, the convergence
\begin{align*}
  \lim_{n\to\infty}  \frac1{u_n^2\bar{F}(u_n)} \esp\left[ \{s^{-1}X_h\ind{X_0>su_n} -t^{-1}X_{h}\ind{X_0>tu_n} \}^2\right]
  & = \frac{\alpha\esp[\Theta_h^2]}{\alpha-2} (s^{-\alpha}-t^{-\alpha})
\end{align*}
is uniform on $[s_0,t_0]^2$. Thus (\ref{eq:continuite-pour-notre-theorem}) holds.
\subsubsection{Proof of \Cref{theo:fclt-edfstp}}
\label{sec:proof-fclt-edfstp}
Consider sets
\begin{align*}
  C_{s,y}(\bx) = \{\bx\in\Rset^{h+1}:|x_0|>s,x_h\leq |x_0|y\} \; .
\end{align*}
for $s_0\leq s \leq t_0$ and $y\in [a,b]$.
The fidi convergence on the class  $\{\1{C_{s,y}}, s_0\leq s \leq t_0, y\in [a,b]\}$ is a consequence of \Cref{theo:fidi}. We only need to prove tightness over $s \in [s_0,t_0]$ at one
point $y$ in order to conclude the tightness over a finite collection of points $y_1,.\ldots, y_k$, $k\ge0$. Define the linearly ordered class
  $\mcg = \{\1{C_{s,y}}, s_0\leq s \leq t_0\}$. The
class $\mcg$ is pointwise separable and its envelope function
 $\1{C_{s_0,y}}$   is in $\mcl_0$. We now
prove that $(\mcg,\rho_h)$ is totally bounded.  For $s_0\leq s<t$, we have
\begin{align}
  \rho_h^2((\1{C_{s,y}}-\1{C_{t,y}})^2)  & = \{s^{-\alpha}  + t^{-\alpha}  - 2 t^{-\alpha} \}L_h(y) \nonumber \\
  & \leq s^{-\alpha}  + t^{-\alpha}  - 2 t^{-\alpha} \leq \alpha s_0^{-\alpha-1} |s-t| \; .  \label{eq:bound-on-norm-for-spectral}
\end{align}
Thus the class $(\mcg,\rho_h)$ is totally bounded.  Moreover, the convergence
\begin{multline*}
  \lim_{n\to\infty}  \frac1{\bar{F}(u_n)}
   \esp\left[\left(\ind{|X_0|>su_n}\ind{X_h\leq |X_0|y} - \ind{|X_0|>tu_n}\ind{X_h\leq |X_0|y'} \right)^2 \right]  \\
   = \{s^{-\alpha}  + t^{-\alpha}  - 2 (s\vee t)^{-\alpha}\} L_h(y)
\end{multline*}
is uniform on compact sets $[s_0,t_0]$ because of
monotonicity. Therefore~(\ref{eq:continuite-pour-notre-theorem}) holds.

\section{Proof of \Cref{prop:counterexample}}
\label{sec:proof-counter}
  Let $N_n$ be the number of returns to the state~1 before time $n$, that is
\begin{align*}
  N_n = \sum_{j=0}^n \ind{X_j=1} \; .
\end{align*}
Set also $\sigma_{-1} = -\infty$, { $\sigma_0=X_0-1$ and $\sigma_{j}=X_0-1+\sum_{k=1}^{j}Z_{\sigma_{k-1}+1}$} for $j\geq1$.  Then,
$\{N_n\}$ is the counting process associated to the delayed renewal process $\{\sigma_n\}$. That is, for
$n,k\geq0$,
\begin{align*}
  N_n = k \Leftrightarrow \sigma_{k-1} \leq n < \sigma_k \; ,
\end{align*}
Since $\esp[Z_0]<\infty$, setting $\lambda=1/\esp[Z_0]$, we have $N_n/n\to \lambda$ a.s.  With this
notation, we have, for every $s>0$,
\begin{align}
  \label{eq:decomposition-cycle}
  \sum_{j=0}^n \ind{X_j>u_ns} & = (X_0-[u_ns])_+ + \sum_{j=1}^{N_n} ({Z_{\sigma_{j-1}+1}}-[u_ns])_+ + \varsigma_n \; ,
\end{align}
where $\varsigma_n = (n-\sigma_{N_n})\wedge(Z_{N_n}-[u_ns])_+$  is a correcting term accounting for the possibly incomplete last portion of
the path. Since $\varsigma_n=O_P(1)$, it does not play any role in the asymptotics.
\begin{enumerate}[$\bullet$,wide=0pt]
\item  Consider the case $\lim_{n\to\infty}n\tail{F}_Z(u_n) = 0$. Then, for an integer $m>\lambda$,
  \begin{align*}
    \pr\left(\sum_{j=1}^{N_n} (Z_j-[u_ns])_+ \ne 0\right) & \leq \pr(N_n >m n) + \pr \left(
      \sum_{j=1}^{m n}  (Z_j-[u_ns])_+ \ne 0\right) \\
    & \leq \pr(N_n >m n) + \pr( \exists j \in \{1,\dots,mn\}, Z_j>[u_ns]) \\
    & \leq \pr(N_n >m n) + mn\tail{F}_Z([u_ns]) \to 0 \;,\qquad n\to \infty .
  \end{align*}
  This proves our first claim.
\end{enumerate}
We proceed with the case $\lim_{n\to\infty}n\tail{F}_Z(u_n)=\infty$. Using (\ref{eq:X-Z}) and
(\ref{eq:decomposition-cycle}) we have
\begin{multline}
  \sum_{j=0}^n \{\ind{X_j>u_ns} - \pr(X_0>u_ns)\} = \sum_{j=1}^{N_n} \{(Z_{{\sigma_{j-1}+1}}-[u_ns])_+ - \esp[(Z_0-[u_ns])_+]\} \\
  +  (X_0-[u_ns])_+ + \varsigma_n + \{N_n-\lambda n) \esp[(Z_0-[u_ns])_+] \; . \label{eq:decomp}
\end{multline}
\begin{enumerate}[$\bullet$,wide=0pt]
\item Consider the case $n\tail{F}_Z(u_n) \to \infty$ and $\tailz\in(1,2)$.  Since
  $\lim_{n\to\infty}\esp[(Z_0-[u_ns])_+]=0$, we obtain, for every $s>0$,
\begin{multline*}
  a_n^{-1}  \sum_{j=0}^n  \{\ind{X_j>u_ns} - \pr(X_0>u_ns)\} \\
   = a_n^{-1} \sum_{j=1}^{N_n} \{Z_{{\sigma_{j-1}+1}} - \esp[Z_0]\} { -  }  a_n^{-1} \sum_{j=1}^{N_n} \{Z_{{\sigma_{j-1}+1}}\wedge [u_ns] -
  \esp[Z_0\wedge[u_ns]]\} + o_P(1) \; .
\end{multline*}
  By regular variation of $\tail{F}_Z$, we obtain
  \begin{align*}
    \var\left(\sum_{j=1}^n \{Z_j\wedge [u_ns] -\esp[(Z_0\wedge [u_ns])]\}\right) =
    O(u_n^2n\tail{F}_Z(u_n)) \; .
  \end{align*}The regular variation of $\tail{F}_Z$ and the conditions $n\tail{F}_Z(u_n)\to\infty$ and
  $n\tail{F}_Z(a_n)\to1$ imply that $u_n/a_n\to0$.  Define $h(x) = x\sqrt{\tail{F}_Z(x)}$. The function  $h$
  is regularly varying at infinity with index $1-\tailz/2>0$ and thus
  \begin{align*}
    \lim_{n\to\infty} \frac{u_n \sqrt{n\tail{F}_Z(u_n)} }{a_n} = \lim_{n\to\infty} \frac{u_n
      \sqrt{\tail{F}_Z(u_n)} }{a_n\sqrt{\tail{F}_Z(a_n)}} = \lim_{n\to\infty} \frac{h(u_n)}{h(a_n)}
    = 0 \; .
  \end{align*}
  This yields
  \begin{align*}
    a_n^{-1}  \sum_{j=0}^n  \{\ind{X_j>u_ns}  - \pr(X_0>u_ns)\}
    & = a_n^{-1} \sum_{j=1}^{N_n} \{Z_{{\sigma_{j-1}+1}} - \esp[Z_0]\} + o_P(1) \; ,
  \end{align*}
  where the $o_P(1)$ term is locally uniform with respect to $s>0$.  Since the distribution of $Z_0$
  is in the domain of attraction of the $\tailz$-stable law and the sequence $\{Z_j\}$ is \iid,
   $\{a_n^{-1} (\sigma_{[ns]}-\lambda^{-1}ns),s>0\} \Rightarrow\Lambda$, where $\Lambda$ is a mean zero,
  totally skewed to the right $\tailz$-stable L\'evy process, and the convergence holds with respect
  to the $J_1$ topology on compact sets of $(0,\infty)$.  Since $\lim_{n\to\infty} N_n/n= \lambda$
  a.s., and since a L\'evy process is stochastically continuous, this yields, by
  \cite[Proposition~13.2.1]{whitt:2002},
  \begin{align*}
    a_n^{-1} \sum_{j=1}^{N_n} \{Z_{{\sigma_{j-1}+1}} - \esp[Z_0]\} \convdistr \Lambda(\lambda)\;,\qquad n\to \infty.
  \end{align*}
  This proves the second claim.
\item Consider now the case $\tailz>2$. In that case, Vervaat's Lemma implies that $(N_n-\lambda
  n)/\sqrt n$ converges weakly to a gaussian distribution. Thus, (\ref{eq:decomp}) combined with
  $\esp[(Z_0-[u_ns])_+]=O(u_n\tail{F}_Z(u_n))$, yields
  \begin{align*}
    (N_n-\lambda n)\esp[(Z_0-[u_ns])_+] = O_P(u_n\tail{F}_Z(u_n)\sqrt{n}) \; .
  \end{align*}
  Next, we apply the Lindeberg central limit theorem for triangular arrays of independent random
  variables to prove that
  \begin{align*}
    \frac1{u_n\sqrt{n\tail{F}_Z(u_n)}} \sum_{j=1}^{n} \{(Z_{{{\sigma_{j-1}+1}}}-[u_ns])_+-\esp[(Z_0-[u_ns])_+]\}
      \convdistr N\left(0,\frac{2s^{1-\alpha}}{ \alpha (\alpha-1)} \right) \; .
  \end{align*}
  By regular variation of $\tail{F}_Z$, we have, for all $\delta\in[2,\tailz)$,
  \begin{align*}
    \esp[(Z_0-u_ns)_+^\delta] \sim C_\delta  u_n^\delta \tail{F}_Z(u_n) s^{\delta-\tailz} \; ,
  \end{align*}
  with $C_\delta = \delta \int_1^\infty (z-1)^{\delta-1} z^{-\tailz} \rmd z$. Set
  \begin{align*}
    Y_{n,j}(s) = \frac1{u_n\sqrt{n\tail{F}_Z(u_n)}} \{(Z_j-[u_ns])_+-\esp[(Z_0-[u_ns])_+]\} \; .
  \end{align*}
  The previous computations yield, for $\delta\in(2,\tailz)$ and $s>0$,
  \begin{align*}
    \lim_{n\to\infty} n\var(Y_{n,1}(s)) & =  \frac{2 s^{2-\tailz}}{  (\tailz-1)(\tailz-2)}  \; , \\
    n\esp[|Y_{n,1}|^\delta] & = O\left( \frac{nu_n^\delta\tail{F}_Z(u_n)}{u_n^\delta
        \{n\tail{F}_Z(u_n)\}^{\delta/2}} \right) = O\left( \{n\tail{F}_Z(u_n)\}^{1-\delta/2}\right) =o(1) \; .
  \end{align*}
  We conclude that the Lindeberg central limit theorem holds. Convergence of the finite dimensional
  distribution is done along the same lines. Tightness with respect to the $J_1$ topology on
  $(0,\infty)$ is proved by applying \cite[Theorem~13.5]{billingsley:1999}.
\end{enumerate}

\appendix

\section{Convergence in $\ell_\infty$}\label{sec:technical-tools}
\begin{theorem}[{\citet[Theorem 3.7.23]{gine:nickl:2016}}]
  \label{theo:GN3.7.23}
  Let $\{\mathbb{Z}_n,n \in\Nset\}$, be a sequence of processes with values in $\ell_\infty(\mcf)$. Then the
  following statements are equivalent.
  \begin{enumerate}[(i)]
  \item \label{item:asymequi} The finite dimensional distributions of the processes $\mathbb{Z}_n$
    converge in law and there exists a pseudometric $\rho$ on $\mcf$ such that  $(\mcf, \rho)$ is totally
    bounded and for all $\epsilon>0$,
    \begin{align}
      \label{eq:asymptotic-equicontinuity}
      \lim_{\delta\to0} \limsup_{n\to\infty} \pr^*\left(\sup_{\rho(f,g)<\delta} |\mathbb{Z}_n(f)-\mathbb{Z}_n(g)| > \epsilon\right) = 0 \; .
    \end{align}
  \item \label{item:weak-conv-linfty} There exists a process $\mathbb{Z}$ whose law is a tight Borel
    probability measure on $\ell_\infty(\mcf)$ and such that $\mathbb{Z}_n \convweak \mathbb{Z}$ in
    $\ell_\infty(\mcf)$.
  \end{enumerate}
  Moreover, if \ref{item:asymequi} holds, then the process $\mathbb{Z}$ in
  \ref{item:weak-conv-linfty} has a version with bounded uniformly continuous paths for $\rho$.
\end{theorem}

The following result provides a sufficient condition for (\ref{eq:asymptotic-equicontinuity}) above.
Let $\{X_{n,i},1\leq i \leq m_n\}$, $n\geq1$, be an array of row-wise \iid\ random elements in a
measurable space $(\mathsf{X},\mathcal{X})$ and define $Z_{n,i}(f) = f(X_{n,i})$,
$f\in\mcf$. {Let $a_n$ be a non decreasing sequence and}  $\mcf$ be a set of measurable
functions defined on $\mathsf{X}$. Define the random pseudometric $d_n$ on $\mcf$ by
\begin{align*}
  d_n^2(f,g) = {\frac1{a_n^2}}\sum_{i=1}^{m_n} \{f(X_{n,i})-g(X_{n,i})\}^2\; , \ \ f,g\in \mcf\;.
\end{align*}
Let $N(\epsilon,\mcf,d_n)$ be the minimum number of balls in the pseudometric $d_n$ needed to
cover~$\mcf$.  Let $\mathbb{Z}_n$ be the empirical process defined by
\begin{align*}
  \mathbb{Z}_n(f) = \frac1{a_n} \sum_{i=1}^{m_n} \{f(X_{n,i})-\esp[f(X_{n,i})]\}  \; , \ \ f \in \mcf \; .
\end{align*}
Define finally the sup-norm $\|H\|_\mcf = \sup_{f\in\mcf} |H(f)|$ for any functional $H$ on
$\mcf$. If $\mcf$ is a pseudometric space and $H$ is measurable on $\mcf$ then the separability of
$\mcf$ implies that $\|H\|_\mcf$ is measurable.
\begin{theorem}[Adapted from {\citet[Theorem~2.11.1]{vandervaart:wellner:1996}}]
  \label{theo:VW2.11.1}
  Assume that the pseudometric space  $\mcf$ is totally bounded and pointwise separable.
  \begin{enumerate}[(i)]
  \item For all $\eta>0$,
    \begin{align}
      \label{eq:lindeberg-envelope}
      \lim_{n\to\infty} a_n^{-2} {m_n} \esp[\|Z_{n,1}\|^2_\mcf\ind{\|Z_{n,1}\|_\mcf>\eta a_n}] = 0 \; .
    \end{align}
  \item For every sequence $\{\delta_n\}$ which decreases to zero,
    \begin{align}
      \lim_{n\to\infty} \sup_{f,g\in\mcf\atop \rho(f,g)\leq\delta_n} \esp[d_n^2(f,g)] = 0 \; , \label{eq:continuite-L2}  \\
      \int_0^{\delta_n} \sqrt{\log N(\epsilon,\mcf,d_n)} \rmd \epsilon \stackrel{\pr}{\longrightarrow} 0  \; . \label{eq:randomentropy}
    \end{align}
  \end{enumerate}
  Then $\mathbb{Z}_n$ is asymptotically $\rho$-equicontinuous, \ie\ (\ref{eq:asymptotic-equicontinuity}) holds. .
\end{theorem}

\paragraph{Acknowledgements} The research of Rafal Kulik was supported by the NSERC grant 210532-170699-2001.  The
research of Philippe Soulier and Olivier Wintenberger was partially supported by the
ANR14-CE20-0006-01 project AMERISKA network. Philippe Soulier also acknowledges support from the
LABEX MME-DII.

\bibliographystyle{apalike}

\end{document}

%% file: commands.tex
\def\constant{\mathrm{cst}}

\def\dfstp{L} %%%% empirical distribution function of the spectral tail process
\def\limdfstp{\mathbb{L}} %%%% limiting distribution of the empirical distribution function of the spectral tail process

\def\HH{H} %%%% cdf of the norm of $\xi_0$.

\def\nuxi{\mu} %%%% exponent measure of the vector $\xi$

\newcommand\1[1]{\mathbbm{1}_{#1}}
\newcommand\ind[1]{\mathbbm{1}{\{#1\}}}

\def\Nset{\mathbb{N}}
\def\Zset{\mathbb{Z}}
\def\Rset{\mathbb{R}}

\def\Eset{\mathsf{E}}

\newcommand{\bu}{\boldsymbol{u}}

\newcommand{\bx}{\boldsymbol{x}}

\newcommand{\bz}{\boldsymbol{z}}
\newcommand{\bX}{\boldsymbol{X}}

\newcommand{\vectorbold}[1]{\boldsymbol{#1}}

\newcommand\mby{\mathbb{Y}}
\newcommand\mbz{\mathbb{Z}}

\newcommand\blockX{\mathbb{X}}

\newcommand\classF[1]{\widehat{\mathcal{#1}}}

\newcommand\sphere[1]{\mathbb{S}^{#1}}

\def\ellinfty{\ell_\infty}

\def\limitmap{\varphi}

\newcommandx{\numult}[2][2=]{{\boldsymbol{\nu}}_{\vectorbold{#1}_{#2}}} %%%% same with an optional argument for index of vector
\newcommand{\numultseq}[1]{{\boldsymbol{\nu}}_{#1}} %%%% exponent measure of a stationary sequence

\def\convdistr{\stackrel{\mbox{\tiny\rm d}}{\to}} % finite dimensional weak convergence
\def\fidi{\stackrel{\mbox{\tiny\rm fi.di.}}{\longrightarrow}} % finite dimensional weak convergence
\def\convprob{\stackrel{\mbox{\small\tiny p}}{\to}}

\def\convweak{\Rightarrow}   %%%%%   What is it ???
\def\rme{\mathrm{e}}

\def\rmd{\mathrm{d}} % Variable in integrals

\def\esp{\mathbb E}
\def\pr{\mathbb P}
\def\var{\mathrm{var}}
\def\cov{\mathrm{cov}}

\newcommand{\CTE}{{\rm CTE}}
\newcommand\tail[1]{\bar{#1}}

\def\mce{\mathcal E}
\def\mcf{\mathcal F}
\def\mcg{\mathcal G}
\def\mch{\mathcal H}

\def\mcl{\mathcal L}
\def\mcm{\mathcal M}

\newcommandx\orderstat[2][1=X]{#1_{#2}}

\newcommand\uncompact[1]{\overline{\Rset}^{#1}\setminus\{\vectorbold0\}}

\newcommand{\sequence}[3]{\{#1_#2,#2\in#3\}}
\newcommand{\sequenceshort}[2]{\{#1_#2\}}

\def\iid{i.i.d.}
\def\ie{i.e.}
\def\eg{e.g.}
\def\wrt{with respect to}

\def\tailz{\beta} %%%% tail index of $Z$ in the counterexample

\newcommand\TEP{\mathbb{G}} % Univariate TEP

\newcommandx\TEPmult[1][1=]{\mathbb{M}^{#1}} % Multiparameter TEP -> old notation $G$

\newcommand\return[1]{\tau_{#1}}